\documentclass[twoside,11pt,reqno]{amsart}
\usepackage{amsmath,amssymb,amscd,mathrsfs,diagrams}

\makeatletter

\@addtoreset{equation}{section}

\newtheorem{Proposition}{Proposition}[section]
\newtheorem{Lemma}[Proposition]{Lemma}
\newtheorem{Theorem}[Proposition]{Theorem}
\newtheorem{Corollary}[Proposition]{Corollary}
\newtheorem{Definition}[Proposition]{Definition}
\newtheorem{Remark}[Proposition]{Remark}

\newtheorem{Example}[Proposition]{Example}

\newbox\squ  
\setbox\squ=\hbox{\vrule width.6pt
   \vbox{\hrule height.6pt width.4em\kern1ex
         \hrule height.6pt}%
   \vrule width.6pt}

\def\cfs{,\C}

\def\anti{(\operatorname{-anti})}
\def\inv{(\operatorname{-inv})}

\def\sgn{\operatorname{sgn}}
\def\Perv{\operatorname{Perv}}

\def\Gr{\operatorname{Gr}}
\def\Col{\operatorname{Col}}
\def\Std{\operatorname{Std}}
\def\Perv{\operatorname{Perv}}
\def\Par{\mathscr P}

\def\deg{\operatorname{deg}}
\def\id{\operatorname{id}}

\def\Id{\operatorname{Id}}
\def\C{{\mathbb C}}

\def\Z{{\mathbb Z}}

\def\0{{\bar 0}}
\def\1{{\bar 1}}
\def\pr{{\operatorname{pr}}}
\def\inv{{\operatorname{-inv}}}
\def\anti{{\operatorname{-anti}}}

\def\R{{\mathtt R}}

\def\hom{{\operatorname{Hom}}}

\def\End{{\operatorname{End}}}

\def\res{{\operatorname{res}}}

\def\im{{\operatorname{im}}}

\def\bj{{\underbar{j}}}

\def\bz{\hbox{\boldmath{$0$}}}
\def\sbz{\hbox{\boldmath{$\scriptstyle{0}$}}}

\def\eps{{\varepsilon}}
\def\phi{{\varphi}}

\def\la{{\lambda}}
\def\La{{\Lambda}}

\def\underbar{\mathpalette\@underbar}
\def\@underbar#1#2{\settowidth{\@tempdimb}{$#1#2$}\@tempdimb=0.8\@tempdimb
                   \ooalign{$#1#2$\crcr%
                         \hfil\rule[-.5mm]{\@tempdimb}{.4pt}\hfil}}

\newdimen\hoogte    \hoogte=8pt    
\newdimen\breedte   \breedte=8pt   
\newdimen\dikte     \dikte=0.5pt    

\newenvironment{young}{\begingroup
       \def\vr{\vrule height0.8\hoogte width\dikte depth 0.2\hoogte}
       \def\fbox##1{\vbox{\offinterlineskip
                    \hrule height\dikte
                    \hbox to \breedte{\vr\hfill##1\hfill\vr}
                    \hrule height\dikte}}
       \vbox\bgroup \offinterlineskip \tabskip=-\dikte \lineskip=-\dikte
            \halign\bgroup &\fbox{##\unskip}\unskip  \crcr }
       {\egroup\egroup\endgroup}
\def\youngdiagram#1{\relax\ifmmode\vcenter{\,\begin{young}#1\end{young}\,}\else%
              $\vcenter{\,\begin{young}#1\end{young}\,}$\fi}

\newdimen\Hoogte    \Hoogte=14pt    
\newdimen\Breedte   \Breedte=14pt   
\newdimen\Dikte     \Dikte=0.8pt    

\newenvironment{Young}{\begingroup
       \def\vr{\vrule height0.8\Hoogte width\Dikte depth 0.2\Hoogte}
       \def\fbox##1{\vbox{\offinterlineskip
                    \hrule height\Dikte
                    \hbox to \Breedte{\vr\hfill##1\hfill\vr}
                    \hrule height\Dikte}}
       \vbox\bgroup \offinterlineskip \tabskip=-\Dikte \lineskip=-\Dikte
            \halign\bgroup &\fbox{##\unskip}\unskip  \crcr }
       {\egroup\egroup\endgroup}
\def\Youngdiagram#1{\relax\ifmmode\vcenter{\,\begin{Young}#1\end{Young}\,}\else%
              $\vcenter{\,\begin{Young}#1\end{Young}\,}$\fi}

\begin{document}

\title[Spaltenstein varieties]{Cohomology of Spaltenstein varieties}
\author{\sc Jonathan Brundan and Victor Ostrik}

\thanks{Research supported in part by NSF grants DMS-0654147 and DMS-0602263.}
\address{Department of Mathematics, University of Oregon, Eugene, Oregon, USA.}
\email{brundan@uoregon.edu, vostrik@uoregon.edu}
\maketitle


\vspace{3mm}

\begin{abstract}
We give a presentation for the cohomology algebra of the
Spaltenstein variety of all partial flags annihilated by a fixed nilpotent
matrix, generalizing the description of the cohomology algebra of the 
Springer fiber found by De Concini, Procesi and Tanisaki.
\end{abstract}

\section{Introduction}

Throughout the article, we fix integers $n, d \geq 0$ and
write $\Lambda(n,d)$ for the set of all $n$-part
{\em compositions} 
$\mu = (\mu_1,\dots,\mu_n)$ 
of $d$, so the $\mu_i$'s are non-negative integers summing to $d$.
Let $\Lambda^+(n,d) \subseteq \Lambda(n,d)$ denote the
$n$-part {\em partitions} of $d$, i.e. the $\lambda \in \Lambda(n,d)$ satisfying
$\lambda_1 \geq \cdots\geq \lambda_n$.
For $\lambda \in \Lambda^+(n,d)$, we write $\lambda^T$ for the transpose partition (which may have more than $n$ non-zero parts). 
Let $X$ be the complex projective variety of
flags $(U_0,\dots,U_d)$ in $\C^d$, so 
\begin{equation*}
\{0\} = U_0 < U_1 < \cdots < U_d = \C^d,
\qquad
\dim U_i / U_{i-1} = 1.
\end{equation*}
By a classical result of Borel, the cohomology algebra $H^*(X\cfs)$ 
is isomorphic to the {\em coinvariant algebra}, which is the quotient $P / I$
of the polynomial algebra 
$P := \C[x_1,\dots,x_d]$, graded so each $x_i$ is in degree 2,
by the ideal $I$ generated by the homogeneous symmetric polynomials 
of positive degree.
To fix a specific isomorphism,
let $\widetilde{U}_i$ 
be the sub-bundle of the trivial vector bundle $\C^d \times X
\rightarrow X$ 
having fiber $U_i$ over the point
$(U_0,\dots,U_d) \in X$.
Then there is a unique graded algebra
isomorphism \begin{equation*}
\phi:C \stackrel{\sim}{\longrightarrow} H^*(X\cfs)
\end{equation*}
with $\phi(x_i)= - c_1(\widetilde{U}_i / \widetilde{U}_{i-1}) \in H^2(X\cfs)$
for each $i=1,\dots,d$;
see e.g. \cite[$\S$10.2, Proposition 3]{Ful}.

Associated to $\lambda \in \Lambda^+(n,d)$, we have the
{\em Springer fiber} $X^\lambda$, which is the closed subvariety of $X$
consisting of
all flags annihilated by the nilpotent matrix $x^\lambda$ 
of Jordan type $\lambda^T$, so
\begin{equation*}
X^\lambda = \{(U_0,\dots,U_d) \in X\:|\: x^\lambda U_i \subseteq U_{i-1}\text{ for each }i=1,\dots,d\}.
\end{equation*}
The pull-back homomorphism 
$i^*:H^*(X\cfs) \rightarrow H^*(X^\lambda\cfs)$
arising from the 
inclusion $i:X^\lambda \hookrightarrow X$
is surjective. In \cite{DP}, De Concini and Procesi computed generators for
the ideal of $C$ that maps to $\ker i^*$ under the isomorphism $\phi$, thus obtaining an explicit presentation for $H^*(X^\lambda\cfs)$.
Soon after that, a slightly simplified description was given by Tanisaki in \cite{T}, as follows.
Let 
$C^\lambda := P / I^\lambda$ where
$I^\lambda$ is the ideal generated by the elementary symmetric functions
\begin{equation*}
\left
\{e_r(x_{i_1},\dots,x_{i_m})\:\bigg|\:
\begin{array}{l}
m \geq 1, 1 \leq i_1<\cdots<i_m \leq d,\\
r > m - \lambda_{d-m+1}-\cdots-\lambda_n
\end{array}
\right\}.
\end{equation*}
Then there is a unique isomorphism
$\bar\phi:C^\lambda\stackrel{\sim}{\rightarrow} H^*(X^\lambda\cfs)$
making the diagram
\begin{equation}\label{diag1}
\begin{CD}
C&@>\phi>>&H^*(X\cfs)\\
@VVV&&@VVi^* V\\
C^\lambda& @>\bar\phi>>&H^*(X^\lambda\cfs)
\end{CD}
\end{equation}
commute, where the left hand map is the canonical quotient map 
which makes sense because
$I \subseteq I^\lambda$.

The symmetric group $S_d$ acts on $P$, hence also on
the quotients $C$ and $C^\lambda$, by algebra automorphisms
so that $w \cdot x_i = x_{w(i)}$ for $w \in S_d$.
Using the isomorphism $\phi$, we get induced an $S_d$-action 
on $H^*(X\cfs)$, which has (at least) two geometric definitions:
the classical one
as in \cite[$\S$13.1]{Jantzen}, and another 
more sophisticated construction originating in the remarkable work of Springer \cite{S1,S2}
(see also \cite[$\S$13.6]{Jantzen}).
Springer's construction yields also an action
of $S_d$ on $H^*(X^\lambda\cfs)$, uniquely determined by the property that 
the surjection $i^*$ is $S_d$-equivariant (see \cite[Theorem 1.1]{HS}
or \cite[$\S$13.13]{Jantzen}).
This plays an essential role in
the derivation of the De Concini-Procesi-Tanisaki presentation.
It is also one of the reasons the cohomology algebras $H^*(X^\lambda\cfs)$ are so interesting:
the top degree cohomology is isomorphic as an $S_d$-module to the irreducible 
representation $S(\lambda^T)$
indexed by the transpose partition $\lambda^T$ (see \cite[Proposition 2.7]{HS}
or \cite[$\S$13.16]{Jantzen}).

The main goal of this article is to prove a parabolic analogue of
(\ref{diag1}).
Fix $\lambda \in \Lambda^+(n,d)$
as before
and also $\mu \in \Lambda(n,d)$.
We are going to replace the flag variety $X$ by
the {\em partial flag variety} $X_\mu$ consisting
of flags
$(V_0,\dots,V_n)$ of type $\mu$ in $\C^d$, so 
\begin{equation*}
\{0\} = V_0 \leq V_1 \leq\cdots\leq V_n = \C^d,
\qquad
\dim V_i / V_{i-1} = \mu_i,
\end{equation*}
and to replace the Springer fiber $X^\lambda$
by the {\em Spaltenstein variety} 
\begin{equation*}
X^\lambda_\mu = \{(V_0,\dots,V_n) \in X_\mu\:|\:x^\lambda V_i \subseteq V_{i-1}\text{ for each }i=1,\dots,n\}.
\end{equation*}
Our main result gives an explicit presentation for the cohomology algebra
of the Spaltenstein variety.

To formulate this,
we must first recall the classical description of
the cohomology algebra $H^*(X_\mu\cfs)$.
Let $P_\mu := P^{S_\mu\inv}$ be the subalgebra of $P$
consisting of all $S_\mu$-invariants, where $S_\mu$ denotes
the parabolic subgroup $S_{\mu_1}\times\cdots\times S_{\mu_n}$ of 
$S_d$. For 
$1 \leq i_1 < \cdots < i_m \leq n$ and $r \geq 1$, we let
$e_r(\mu;i_1,\dots,i_m)$ and $h_r(\mu;i_1,\dots,i_m)$
denote 
the $r$th elementary and complete symmetric functions
in the variables $X_{i_1}\cup\cdots\cup X_{i_m}$, where
$$
X_j = \{x_k\:|\:\mu_1+\cdots+\mu_{j-1}+1 \leq k \leq \mu_1+\cdots+\mu_j\}.
$$
We interpret 
$e_r(\mu;i_1,\dots,i_m)$ and $h_r(\mu;i_1,\dots,i_m)$
as $1$ if $r = 0$ and
as $0$ if $r < 0$.
Note also that
\begin{align}\label{easyexp2}
e_r(\mu;i_1,\dots,i_m) = \sum_{r_1+\cdots+r_m=r}
e_{r_1}(\mu;i_1) \cdots e_{r_m}(\mu;i_m),\\
h_r(\mu;i_1,\dots,i_m) = \sum_{r_1+\cdots+r_m=r}
h_{r_1}(\mu;i_1) \cdots h_{r_m}(\mu;i_m).
\label{easyexp}
\end{align}
The algebra $P_\mu$ is freely generated either by 
$\{e_r(\mu;i)\:|\:1 \leq i \leq n, 1\leq r \leq \mu_i\}$
or by
$\{h_r(\mu;i)\:|\:1 \leq i \leq n, 1\leq r \leq \mu_i\}$.
Let $I_\mu$ be the ideal of $P_\mu$ generated by 
homogeneous symmetric polynomials of positive degree. Set $C_\mu := P_\mu / I_\mu$.
Then there is a unique isomorphism
\begin{equation*}
\psi:C_\mu \stackrel{\sim}{\longrightarrow} H^*(X_\mu\cfs)
\end{equation*}
with $\psi(e_r(\mu;i)) = (-1)^r c_r(\widetilde V_i / \widetilde V_{i-1})$
for each $i=1,\dots,n$ and $r > 0$, where
 $\widetilde V_i$ denotes the sub-bundle of the trivial
vector bundle $\C^d\times X_\mu \rightarrow X_\mu$ with fiber
$V_i$ over $(V_0,\dots,V_n) \in X_\mu$. 

Let $C^\lambda_\mu := P_\mu / I^\lambda_\mu$ where 
$I^\lambda_\mu$ is the ideal of $P_\mu$ generated by the elements
\begin{equation}\label{rel1}
\left\{h_r(\mu;i_1,\dots,i_{m})\:\bigg|
\begin{array}{l}
m \geq 1, \:\:1 \leq i_1<\cdots< i_{m} \leq n,\\
r > \lambda_1+\cdots+\lambda_{m} - \mu_{i_1}-\cdots-
\mu_{i_{m}}\end{array}\right\}.
\end{equation}
Equivalently (see \cite[Lemma 2.2]{duke}) $I^\lambda_\mu$ is generated by
\begin{equation}\label{rel2}
\left\{e_r(\mu;i_1,\dots,i_{m})\:\Bigg|
\begin{array}{l}
m \geq 1, \:\: 1 \leq i_1<\cdots<i_{m}\leq n,\\
r > \mu_{i_1}+\cdots+\mu_{i_{m}} - \lambda_{l+1}-\cdots-\lambda_n\\
\text{where }l:=\#\{i \:|\: \mu_i > 0, i \neq i_1,\dots,i_m\}
\end{array}\right\},
\end{equation}
from which it is easy to see that $I^\lambda_\mu = I^\lambda$, hence
$C^\lambda_\mu = C^\lambda$, if $\mu$ is regular (all parts $\leq 1$).
Finally let $j:X^\lambda_\mu \hookrightarrow X_\mu$ be the inclusion, and note that the pull-back $j^*:H^*(X_\mu\cfs)\rightarrow H^*(X^\lambda_\mu\cfs)$ is surjective; see $\S$\ref{spaving}.

\begin{Theorem}\label{mt1}
There is a unique isomorphism
$\bar\psi:C^\lambda_\mu \stackrel{\sim}{\rightarrow} H^*(X^\lambda_\mu\cfs)$
making the diagram
\begin{equation}\label{diag2}
\begin{CD}
C_\mu&@> \psi >>&H^*(X_\mu\cfs)\\
@VVV&&@VVj^*V\\
C^\lambda_\mu&@>\bar\psi>>&H^*(X^\lambda_\mu\cfs)
\end{CD}
\end{equation}
commute, where
the left hand map is the canonical quotient map 
which makes sense because $I_\mu \subseteq I^\lambda_\mu$.
\end{Theorem}

The precise form of the relations (\ref{rel1})--(\ref{rel2}) 
was originally worked out in \cite{duke},
which is concerned with the centers of integral
blocks of parabolic category $\mathcal O$
for the general linear Lie algebra. These centers
provide another natural occurrence
of the algebras $C^\lambda_\mu \cong H^*(X^\lambda_\mu\cfs)$;
see also \cite[Theorem
4.1.1]{Stroppel} (which treats regular $\mu$ only) and \cite[Remark
9.10]{BLPPW} (for all $\mu$). 

We next formulate a result which plays an
essential role in the proof of Theorem~\ref{mt1}.
Let $p:X \rightarrow X_\mu$ be the projection
sending $(U_0, \dots, U_d) \in X$
to $(V_0, \dots, V_n) \in X_\mu$ where
$V_i = U_{\mu_1+\cdots+\mu_i}$ for each $i$.
It is classical that the pull-back $p^*$ 
defines a graded algebra isomorphism between
$H^*(X_\mu\cfs)$
and $H^*(X\cfs)^{S_\mu\inv}$. Moreover the diagram
\begin{equation}\label{invi}
\begin{CD}
C_\mu&@>\sim>>&C^{S_\mu\inv}\\
@V\psi VV&&@VV\phi V \\
H^*(X_\mu\cfs)&@>p^*>>&H^*(X\cfs)^{S_\mu\inv}
\end{CD}
\end{equation}
commutes, where the top map is induced by the inclusion 
$P_\mu \hookrightarrow P$.
Using Poincar\'e duality, it also makes sense to 
consider the push-forward
$p_*$ as a map in cohomology.
Let
\begin{equation*}
d_\mu := \dim X - \dim X_\mu = \frac{1}{2}\sum_{i=1}^n \mu_i(\mu_i-1).
\end{equation*}
Then $p_*$
restricts to give an
$H^*(X_\mu\cfs)$-module
{isomorphism}
between the space $H^*(X\cfs)^{S_\mu\anti}$ of
$S_\mu$-anti-invariants 
in $H^*(X\cfs)$
and the regular module $H^*(X_\mu\cfs)$. Thus, there is a commuting diagram
\begin{equation}\label{antii}
\begin{CD}
C^{S_\mu\anti}&@>\sim>>&C_\mu[- 2 d_\mu]\\
@V\phi VV&&@VV\psi V\\
H^*(X\cfs)^{S_\mu\anti}&@>p_*>>&H^*(X_\mu\cfs)[- 2 d_\mu],
\end{CD}
\end{equation}
where for a graded vector space $M$ we write $M[ i ]$
for the graded vector space obtained by shifting the grading down
by $i$, i.e.
$M[ i ]_j = M_{i+j}$.
Viewed as a $C_\mu$-module
via the isomorphism $C_\mu \cong C^{S_\mu\inv}$ from (\ref{invi}),
$C^{S_\mu\anti}$ is free of rank one
generated by the element
\begin{equation*}
\eps_\mu := \frac{1}{|S_\mu|}
\sum_{\substack{1 \leq i < j \leq d,\\\text{same }S_\mu\text{-orbit}}}
 (x_i - x_j).
\end{equation*}
The isomorphism at the top of (\ref{antii})
sends $x \eps_\mu \mapsto x$ for each $x \in C_\mu$;
see e.g. \cite[Lemma 3.2]{duke}.

\begin{Theorem}\label{mt2}
There is a unique homogeneous linear map
$\bar p_*$
making the following diagram commute:
\begin{equation}\label{phys}
\begin{CD}
H^*(X\cfs)&@>p_*>>&H^*(X_\mu\cfs)[- 2d_\mu]\\
@Vi^*VV&&@VVj^*V\\
H^*(X^\lambda\cfs)&@>\bar p_*>>&H^*(X^\lambda_\mu\cfs)[- 2d_\mu].
\end{CD}
\end{equation}
Moreover the restriction of 
$\bar p_*$ 
is an isomorphism
of graded vector spaces
$\bar p_*:H^*(X^\lambda\cfs)^{S_\mu\anti}\stackrel{\sim}{\rightarrow} 
H^*(X^\lambda_\mu\cfs)[- 2d_\mu]$.
\end{Theorem}

The existence of an isomorphism
$H^*(X^\lambda\cfs)^{S_\mu\anti}\cong H^*(X^\lambda_\mu\cfs)[- 2d_\mu]$ was established already by Borho and Macpherson 
\cite[Corollary 3.6(b)]{BM}. 
The point of Theorem~\ref{mt2} is to make this isomorphism canonical; see $\S$\ref{sbm}.
Granted Theorem~\ref{mt2}, 
one possible approach to the proof of Theorem~\ref{mt1}
is sketched in \cite[Remark 4.6]{duke}. 
The proof of
Theorem~\ref{mt1} given in $\S$\ref{stan} below is quite different and 
gives a more natural
explanation of the relations;
it is a generalization of
Tanisaki's original argument in \cite{T}.
In the course of the proof, we also obtain a homogeneous algebraic basis
for $H^*(X^\lambda_\mu\cfs)$ indexed by certain $\lambda$-tableaux of
type $\mu$, which is of independent interest; 
see $\S$\ref{sab}.

We point out finally that
the algebras $H^*(X^\lambda_\mu\cfs)$ arise in
Ginzburg's construction of analogues of the Springer representations
for the general linear Lie algebra; see \cite{Gi, BG}.
In particular, there is a natural way to define an action of
$\mathfrak{gl}_n(\C)$ on the direct sum 
of the $H^*(X^\lambda_\mu\cfs)$'s for all $\mu \in \Lambda(n,d)$, so that the top cohomology
$$
\bigoplus_{\mu \in \Lambda(n,d)} H^{2d_\lambda-2d_\mu}(X^\lambda_\mu\cfs)
$$
is irreducible of highest
weight $\lambda$.
Following
the reformulation by Braverman and Gaitsgory \cite{BG} (see also
\cite[$\S$7]{Gi2}), this action can be constructed by applying
the signed Schur functor $V^{\otimes d} \otimes_{\C S_d} ?$
to the $S_d$-module $H^*(X^\lambda\cfs)$. In more detail,
let $V^{\otimes d}$
denote the $d$th tensor power of the natural $\mathfrak{gl}_n(\C)$-module
viewed as a right $\C S_d$-module 
so that $w \in S_d$ acts by
$$
(v_{1} \otimes \cdots \otimes v_{d})w
= \sgn(w) v_{w(1)} \otimes\cdots\otimes v_{w(d)}.
$$
For $\mu \in \Lambda(n,d)$, the $\mu$-weight space of 
$V^{\otimes d} \otimes_{\C S_d} H^*(X^\lambda\cfs)$ is canonically isomorphic
to $H^*(X^\lambda\cfs)^{S_\mu\anti}$; see \cite[Lemma 3.1]{duke}.
Using also Theorem~\ref{mt2},
we get a canonical vector space isomorphism
\begin{equation}\label{rhs}
V^{\otimes d} \otimes_{\C S_d} H^*(X^\lambda\cfs)
\cong \bigoplus_{\mu \in \Lambda(n,d)} H^*(X^\lambda_\mu\cfs),
\end{equation}
hence can transport the $\mathfrak{gl}_n(\C)$-module structure
from the left to the right hand space.
This yields the desired action.
Using the presentation 
for the algebras
$H^*(X^\lambda_\mu\cfs)$ from Theorem~\ref{mt1},
it is possible to give a purely algebraic construction (no cohomology)
of these representations, in the same spirit as the approach of
Garsia and Procesi to the type A Springer representations in \cite{GP}.
This is pursued further in \cite[$\S$4]{duke}, where one can find 
explicit algebraic formulae for the actions of the Chevalley generators
of $\mathfrak{gl}_n(\C)$ on the space on the right hand side of (\ref{rhs}).

\section{Affine paving}\label{spaving}

By an {\em affine paving} of a variety $M$, we mean
a partition of $M$ into disjoint subsets $M_i$
indexed by some finite poset $(I,\leq)$, such that the following hold
for all $i \in I$:
\begin{itemize}
\item $\bigcup_{j \leq i} M_j$ is closed in $M$;
\item $M_i$ (which is automatically locally closed) is isomorphic
as a variety to $\mathbb{A}^{d_i}$ for some $d_i \geq 0$.
\end{itemize}
Suppose we are given
$\lambda \in \Lambda^+(n,d)$ and
$\mu \in \Lambda(n,d)$.
The Spaltenstein variety $X^\lambda_\mu$ from the introduction 
was defined and studied originally in \cite{Sp}.
In that paper, Spaltenstein constructed an explicit affine paving of
the Springer fiber $X^\lambda$, and deduced from that a parametrization of the
irreducible components of 
$X^\lambda_\mu$. The goal in this section is to 
extend Spaltenstein's argument
slightly to produce an
affine paving of $X^\lambda_\mu$ itself;
actually we 
explain the dual construction 
which produces a more convenient parametrization
from a combinatorial point of view.
This is a widely known piece of folk-lore, but still we could not find 
it explicitly
written in the literature. 

We begin with a few more conventions regarding partitions.
We draw the Young diagram of a partition $\la$ in the usual
English way as in \cite{Mac}.
The sum of the parts of $\la$ is $|\la|$,
and $h(\la)$ denotes the {\em height} of $\la$, that is, the number of non-zero parts.
We write $\Par_k$ for the set of all partitions 
of height at most $k$, and
$\Par_{k,l}$ for the set of all partitions fitting into a $k \times l$-rectangle, i.e. $h(\lambda) \leq k$ and $\lambda_1 \leq l$.
For partitions $\la$ and $\mu$, we write
$\la \subseteq \mu$ if $\la_i \leq \mu_i$ for all $i$.
Also $\leq$ denotes the usual dominance ordering.
For $\mu \in \La(n,d)$, we let $\mu^+ \in \La^+(n,d)$ 
be the unique partition obtained from $\mu$ by rearranging the parts 
in weakly decreasing order.

Given $0 \leq k \leq d$, let $\Gr_{k,d}$ be the
Grassmannian of $k$-dimensional subspaces of $\C^d$.
To a partition $\gamma \in \Par_{k}$, we associate its
{\em column sequence} $(c_1,\dots,c_k)$
defined from
\begin{equation}\label{colseq}
\gamma_{k+1-i} =c_i - i.
\end{equation}
This gives a bijection between
$\Par_{k,d-k}$ and the set of all 
sequences $(c_1,\dots,c_k)$ with $1 \leq c_1 < \cdots < c_k \leq d$.
Given 
a fixed basis $f_1,\dots,f_d$ for $\C^d$
and $\gamma \in \Par_{k,d-k}$,
we have the {\em Schubert variety}
\begin{equation}\label{schub}
Y_\gamma := \{U \in \Gr_{k,d}\:|\:\dim U \cap \langle f_1,\dots,f_{c_i}\rangle \geq i
\text{ for }i=1,\dots,k\},
\end{equation}
where $(c_1,\dots,c_k)$ is the column sequence associated to $\gamma$;
see e.g. \cite[$\S$9.4]{Ful}.
This is an irreducible 
closed subvariety of dimension $|\gamma|$. Also
$Y_\gamma \subseteq Y_{\gamma'}$
if and only if 
$\gamma \subseteq \gamma'$.
The {\em Schubert cells}
$Y_\gamma^\circ := Y_\gamma \setminus \bigcup_{\gamma' \subsetneq \gamma}
Y_{\gamma'}$
give an affine paving of $\Gr_{k,d}$ indexed by the poset
$(\Par_{k,d-k},\subseteq)$.
In terms of coordinates, every $U \in Y_\gamma^\circ$ can be represented
as the span of the vectors
\begin{equation}\label{coords}
f_{c_i}+\sum_{\substack{1 \leq j \leq c_i\\ j \neq c_1,\dots,c_i}} 
a_{i,j} f_j\qquad(i=1,\dots,k)
\end{equation}
for unique $a_{i,j} \in \C$.
In particular, 
$Y_\gamma^\circ \cong \mathbb{A}^{|\gamma|}$.

In this section, we fix the basis $f_1,\dots,f_d$
as follows.
Let $e_1,\dots,e_d$ be the standard basis for $\C^d$.
If we identify the basis vectors $e_1,\dots,e_d$
with the
boxes of the Young diagram of $\lambda$
working up columns starting from the first (leftmost) column,
then the nilpotent matrix $x^\lambda$ of Jordan type $\lambda^T$ is the endomorphism of $\C^d$ sending
a basis vector to the one immediately below it in the diagram, or to zero
if it is at the bottom of its column.
Then
let $f_1,\dots,f_d$ be the basis for $\C^d$ obtained by reading
the boxes of the Young diagram in order along rows starting from the first 
(top) row.
For example:
\begin{equation}\label{piccy}
\lambda = (4,3,2)
\qquad\qquad
 \Youngdiagram{$e_3$&$e_6$&$e_8$&$e_9$\cr $e_2$&$e_5$&$e_7$\cr $e_1$&$e_4$\cr}
\qquad
 \Youngdiagram{$f_1$&$f_2$&$f_3$&$f_4$\cr $f_5$&$f_6$&$f_7$\cr $f_8$&$f_9$\cr}
\qquad\qquad
x^\lambda = \begin{CD}
@VVV
\end{CD}
\end{equation}
Also let $(.,.)$ be the symmetric bilinear form on $\C^d$ such that
$(f_i,f_j) = \delta_{i,j}$.
The following proposition provides the key induction step.

\begin{Proposition}\label{pavingl}
Suppose $n \geq 1$ and that 
 $k := \mu_n, s := \lambda_1$ and $t := \lambda_n$
satisfy $s \geq k \geq t$.
Let 
$\pi:X_\mu \rightarrow \Gr_{k,d}$ 
be the morphism
$(V_0,\dots,V_n) \mapsto V_{n-1}^\perp$.
Let $\overline\mu \in \Lambda(n-1,d-k)$
be obtained from $\mu$ by forgetting the last part $\mu_n$,
let
$\beta := ((s-k)^{k-t}) \in \Par_{k,d-k}$,
and take any $\gamma \in \Par_{k,d-k}$
with column sequence $(c_1,\dots,c_k)$.
\begin{itemize}
\item[(i)]
The restriction of $\pi$ defines a morphism 
$\bar\pi:X^\lambda_\mu \rightarrow Y_\beta.$
\item[(ii)]
There is an isomorphism
of varieties 
$f_\gamma:\pi^{-1}(Y_\gamma^\circ) \stackrel{\sim}{\longrightarrow}
Y_\gamma^\circ \times X_{\overline\mu}.$
\end{itemize}
Assume in addition that $\gamma \subseteq \beta$.
Let $\overline{\lambda} \in \Lambda^+(n-1,d-k)$ be the partition whose Young 
diagram is obtained by removing one box from the bottom of each of the columns
numbered $c_1=1,\dots,c_t=t,c_{t+1},\dots,c_k$ in the Young diagram of $\lambda$, then re-ordering
the columns to get a proper partition shape.
\begin{itemize}
\item[(iii)] For $(V_0,\dots,V_n) \in \pi^{-1}(Y_\gamma^\circ)$,
we have $x^\lambda(V_{n-1}) \subseteq x^\lambda(V_n) \subseteq V_{n-1}$,
and the restriction of $x^\lambda$ to $V_{n-1}$ is
of Jordan type $(\overline{\lambda})^T$.
\item[(iv)]
The isomorphism $f_\gamma$ can be chosen so that it restricts
to an isomorphism
$\bar f_\gamma:\bar\pi^{-1}(Y_\gamma^\circ)\stackrel{\sim}{\longrightarrow}
Y_\gamma^\circ \times X^{\overline\lambda}_{\overline\mu}$.
\end{itemize}
\end{Proposition}

\begin{proof}
Note that $Y_\beta = \{U \in \Gr_{k,d}\:|\:
\langle f_1,\dots,f_t \rangle \subseteq 
U \subseteq \langle f_1,\dots,f_s\rangle\}$.
Take $(V_0,\dots,V_n) \in X^\lambda_\mu$.
Then $x^\lambda(V_n) \subseteq V_{n-1}$
and $(x^\lambda)^{n-1} (V_{n-1}) = \{0\}$,
or equivalently, 
$\langle f_1,\dots,f_t \rangle = 
(\ker\: (x^\lambda)^{n-1})^\perp\subseteq
V_{n-1}^\perp \subseteq (\im\: x^\lambda)^\perp 
= \langle f_1,\dots,f_s\rangle$.
This shows $\pi(X^\lambda_\mu) \subseteq Y_\beta$, establishing (i).

Now we'll prove (ii), (iii) and (iv) all under the assumption that
$\gamma \subseteq \beta$, noting that (ii) for more general $\gamma$
is actually a particular case of (iv) on replacing 
$\lambda$ by the partition of $d$
with just one non-zero part.

The subspace $\langle f_{c_1},\dots,f_{c_k}\rangle$
belongs to $Y_\gamma^\circ$.
Take $U \in Y_\gamma^\circ$
represented by the coordinates $(a_{i,j})$ according to (\ref{coords}).
The non-zero vectors 
of the form 
$(x^\lambda)^i (f_j)$ for $i \geq 0$ and $1 \leq j \leq s$ form a basis
for $\C^d$.
Hence there exists a unique matrix $g(U) \in GL_d(\C)$
such that
\begin{itemize}
\item $g(U)(f_{c_i}) = 
f_{c_i} + \sum_{j \in \{1,\dots,c_i\}\setminus\{c_1,\dots,c_i\}} a_{i,j} f_j$ for each $i=1,\dots,k$;
\item $g(U)(f_j) = f_j$ for each $j \in \{1,\dots,s\}\setminus\{c_1,\dots,c_k\}$;
\item $g(U)(x^\lambda (f_j)) = x^\lambda (g(U)(f_j))$
for any $1 \leq j \leq d$ such that $x^\lambda (f_j) \neq 0$.
\end{itemize}
The transpose matrix $(x^\lambda)^T$ annihilates $f_1,\dots,f_s$,
and $(x^\lambda)^T(x^\lambda(f_j)) = f_j$ whenever
$x^\lambda(f_j) \neq 0$.
Using this we see that $g(U)$ commutes with
$(x^\lambda)^T$, or equivalently, $g(U)^T$ commutes with $x^\lambda$.
Moreover 
$g(U)(\langle f_{c_1},\dots,f_{c_k}\rangle) = U$,
hence $g(U)^T(U^\perp) = \langle f_{c_1},\dots,f_{c_k}\rangle^\perp$.

The restriction of $x^\lambda$ to 
the space $\langle f_{c_1},\dots,f_{c_k}\rangle^\perp$
is of Jordan type $(\overline \lambda)^T$.
Hence we can pick an isomorphism
$\theta:\langle f_{c_1},\dots,f_{c_k} \rangle^\perp
\rightarrow \C^{d-k}$ such that 
$\theta \circ x^\lambda = x^{\overline\lambda} \circ \theta$.
Then define
$f_\gamma:\pi^{-1}(Y_\gamma^\circ) \rightarrow Y_\gamma^\circ \times X_{\overline{\mu}}$ to be the map sending
$(V_0,\dots,V_n) \in \pi^{-1}(Y_\gamma^\circ)$
to $(U, (\overline{V}_0,\dots,\overline{V}_{n-1})) \in Y_\gamma^\circ
\times X_{\overline{\mu}}$, where $U = V_{n-1}^\perp$ and
$\overline{V}_i = \theta(g(U)^T(V_i))$.
This makes sense because $g(U)^T(V_{n-1}) = \langle f_{c_1},\dots,f_{c_k}\rangle^\perp$. It is clear that the map $f_\gamma$ is invertible, hence it is an
isomorphism of varieties as in (ii).

To deduce (iii),
take $(V_0,\dots,V_n) \in \pi^{-1}(Y_\gamma^\circ)$.
We have that $V_{n-1}^\perp
\in Y_\gamma^\circ \subseteq Y_\beta$, hence
$V_{n-1}^\perp \subseteq \langle f_1,\dots,f_s\rangle$.
This shows that
$x^\lambda(V_n) = \langle f_1,\dots,f_s\rangle^\perp \subseteq
V_{n-1}$,
so in particular $x^\lambda$ leaves $V_{n-1}$ invariant.
Moreover multiplication by $g(V_{n-1}^\perp)^T$ is an $x^\lambda$-equivariant
isomorphism between $V_{n-1}$ and
$\langle f_{c_1},\dots,f_{c_k}\rangle^\perp$.
Hence the Jordan type of $x^\lambda$ on $V_{n-1}$ is the same
as its Jordan type on $\langle f_{c_1},\dots,f_{c_k}\rangle^\perp$, namely
$(\overline{\lambda})^T$.

Finally, for (iv),
suppose we are given $(V_0,\dots,V_n) \in \pi^{-1}(Y_\gamma^\circ)$
mapping to $(U, (\overline{V}_0,\dots,\overline{V}_{n-1}))$
under the isomorphism $f_\gamma$.
We need to show that $(V_0,\dots,V_n) \in X^\lambda_\mu$
if and only if $(\overline{V}_0,\dots,\overline{V}_{n-1}) \in X^{\overline{\lambda}}_{\overline{\mu}}$.
By (iii), we know automatically that $x^\lambda(V_n) \subseteq V_{n-1}$.
Hence we just need to check for each $i=1,\dots,n-1$ that
$x^\lambda(V_i) \subseteq V_{i-1}$ if and only if
$x^{\overline{\lambda}}(\overline{V}_i) \subseteq \overline{V}_{i-1}$.
This is clear as $x^{\overline{\lambda}}\circ \theta \circ g(U)^T
= \theta \circ g(U)^T \circ x^\lambda$ on $V_{n-1}$.
\end{proof}

By a {\em column-strict} 
(resp.\ {\em semi-standard}) $\lambda$-tableau 
we mean a filling of the boxes
of the Young diagram of $\lambda$ by integers so that entries are strictly
increasing down columns (resp.\ strictly increasing down columns and 
weakly increasing along rows).
Let $\Col^\lambda_\mu$ (resp.\ $\Std^\lambda_\mu$) denote the set of all
column-strict (resp.\ semi-standard) $\lambda$-tableaux
that have exactly $\mu_i$ entries equal to $i$
for each $i=1,\dots,n$.
The following definition will also be needed in the next section.

\begin{Definition}\label{combdef}\rm
Suppose that $n \geq 1$ and we are given $\mathtt{T} \in \Col^\lambda_\mu$.
The existence of $\mathtt{T}$ implies
that 
the integers $k$, $s$, $t$ defined as in Proposition~\ref{pavingl}
satisfy $s\geq k \geq t$.
We define
$(c_1,\dots,c_k)$,
$\gamma$,
$\overline{\mu}$,
$\overline{\mathtt{T}}$, and
$\overline{\lambda}$
as follows:
\begin{itemize}
\item
$1 \leq c_1 < \cdots < c_k \leq s$ 
index the columns of $\mathtt{T}$ containing entry $n$;
\item
$\gamma \in \Par_{k,d-k}$ is the partition with column 
sequence $(c_1,\dots,c_k)$;
\item
$\overline\mu \in \Lambda(n-1,d-k)$ is defined 
by forgetting the last part of $\mu$;
\item $\overline{\mathtt{T}}$ 
is obtained by removing all boxes labelled by $n$ from $\mathtt{T}$,
then repeatedly 
interchanging columns $i$ and $(i+1)$
whenever 
the $i$th column contains fewer boxes than the $(i+1)$th 
column for some $i$;
\item $\overline\lambda \in \Lambda^+(n-1,d-k)$ is 
the shape of $\overline{\mathtt{T}}$,
so $\overline{\mathtt{T}} \in \Col^{\bar\lambda}_{\bar\mu}$.
\end{itemize}
\end{Definition}

\begin{Example}\label{anex}\rm
If $n=6$
and $\mathtt{T} = \youngdiagram{
$\scriptstyle 2$&$\scriptstyle 1$&$\scriptstyle 2$&$\scriptstyle 2$\cr
$\scriptstyle 3$&$\scriptstyle 2$&$\scriptstyle 4$\cr
$\scriptstyle 4$&$\scriptstyle 4$&$\scriptstyle 6$\cr
$\scriptstyle 6$&$\scriptstyle 5$\cr
}
$ 
then $(c_1,c_2) = (1,3)$ and
$
\overline{\mathtt{T}} = \youngdiagram{
$\scriptstyle 1$&$\scriptstyle 2$&$\scriptstyle 2$&$\scriptstyle 2$\cr
$\scriptstyle 2$&$\scriptstyle 3$&$\scriptstyle 4$\cr
$\scriptstyle 4$&$\scriptstyle 4$\cr
$\scriptstyle 5$\cr
}.
$
\end{Example}

Now we make several definitions by induction on $n$.
Suppose first that $n=d=0$.
We define the {\em degree} $\deg(\mathtt{T})$ of the unique tableau 
$\mathtt{T} \in \Col^\lambda_\mu$  (the empty tableau)
to be zero, let $Y_{\mathtt{T}}^\circ := X^\lambda_\mu$ (which is a 
single point), and let $\preceq$ be the trivial partial order
on $\Col^\lambda_\mu$.
Now suppose that $n \geq 1$, take $\mathtt{T} \in \Col^\lambda_\mu$
and define $\gamma$ and $\overline{\mathtt{T}}$ as in Definition~\ref{combdef}.
Then set
\begin{equation}\label{degdef}
\deg(\mathtt{T}) := |\gamma|+\deg(\overline{\mathtt{T}}).
\end{equation}
For example, the tableau $\mathtt{T}$ from Example~\ref{anex}
is of degree $2$.
Also define a subset
$Y_{\mathtt{T}}^\circ \subseteq X^\lambda_\mu$ 
by setting
\begin{equation}\label{pavingd}
Y_{\mathtt{T}}^\circ := \bar f_\gamma^{-1}(Y_\gamma^\circ \times Y_{\overline{\mathtt{T}}}^\circ),
\end{equation}
where $\bar f_\gamma$ is the map chosen in Proposition~\ref{pavingl}(iv).
Finally define a partial order $\preceq$ on $\Col^\lambda_\mu$ by declaring that
$\mathtt{T} \preceq \mathtt{T}'$ if $\gamma \subsetneq \gamma'$,
or $\gamma = \gamma'$ and
$\overline{\mathtt{T}} \preceq \overline{\mathtt{T}}'$,
where $\gamma'$
and $\overline{\mathtt{T}}'$ are defined using Definition~\ref{combdef} again but starting
from $\mathtt{T}'$.

\begin{Theorem}\label{pavingt}
The sets $Y_{\mathtt{T}}^\circ$
give an affine paving of $X^\lambda_\mu$
indexed by the poset $(\Col^\lambda_\mu, \preceq)$, such that
$Y_{\mathtt{T}}^\circ \cong \mathbb{A}^{\deg(\mathtt{T})}$.
Moreover the complement $X_\mu \setminus X^\lambda_\mu$ also has an affine paving.
\end{Theorem}

\begin{proof}
Consider the map $\bar\pi:X^\lambda_\mu \rightarrow Y_\beta$
from Proposition~\ref{pavingl}(i).
Since $Y_\beta$ is paved by the affine spaces
$Y_\gamma^\circ$ for $\gamma \subseteq \beta$,
it suffices to show that the
inverse images 
$\bar\pi^{-1}(Y_\gamma^\circ)$ all have affine pavings.
This is follows by induction on $n$ using 
Proposition~\ref{pavingl}(iv) and the definition (\ref{pavingd}).
The same induction gives also that
$Y_{\mathtt{T}}^\circ \cong \mathbb{A}^{\deg(\mathtt{T})}$.
Finally, 
to see that the complement $X_\mu \setminus X^\lambda_\mu$
has an affine paving, we use also Proposition~\ref{pavingl}(ii)
to see that
the complement is the disjoint union of the
spaces $\pi^{-1}(Y_\gamma^\circ) \setminus \bar\pi^{-1}(Y_\gamma^\circ)
\cong Y_\gamma^\circ \times (X_{\overline{\mu}} \setminus X_{\overline{\mu}}^{\overline{\lambda}})$ for each $\gamma \subseteq \beta$,
each of which has an affine paving by induction, together
with $\pi^{-1}(Y_\gamma^\circ) \cong Y_\gamma^\circ \times X_{\overline{\mu}}$ for $\gamma \not\subseteq \beta$, 
which have affine pavings too.
\end{proof}

\begin{Corollary}\label{cd}
The cohomology $H^*(X^\lambda_\mu\cfs)$ vanishes in all odd degrees,
and in even degrees $\dim H^{2r}(X^\lambda_\mu\cfs)$ is equal to the number of
$\mathtt{T} \in \Col^\lambda_\mu$ with $\deg(\mathtt{T}) = r$.
Moreover the pull-back
$j^*:H^*(X_\mu\cfs) \rightarrow H^*(X^\lambda_\mu\cfs)$
is surjective.
\end{Corollary}

\begin{proof}
This is a standard consequence of the existence of an affine paving as in Theorem~\ref{pavingt}; see e.g. 
\cite[Corollary 2.3]{HS} or 
the discussion at the bottom of \cite[p.163]{Jantzen}.
\end{proof}

\begin{Corollary}\label{ne}
The variety
$X^\lambda_\mu$ is 
non-empty if and only if $\lambda \geq \mu^+$. 
Assuming that is the case,
$X^\lambda_\mu$ is connected.
\end{Corollary}

\begin{proof}
The first statement
follows because $\Col^\lambda_\mu$ is non-empty if and only if
$\lambda \geq \mu^+$.
For the second statement, we observe that
there is a unique $\mathtt{T} \in \Col^\lambda_\mu$ of degree zero,
hence $\dim H^0(X^\lambda_\mu\cfs)=1$
by Corollary~\ref{cd}.
\end{proof}

Although not really needed in the rest of the article, we end the section by 
explaining for completeness how to recover
Spaltenstein's classification of the irreducible components
of $X^\lambda_\mu$ from Theorem~\ref{pavingt}.
Suppose we are given $(V_0,\dots,V_n) \in X^\lambda_\mu$.
Assuming $n \geq 1$, Proposition~\ref{pavingl}(iii) shows that the restriction of
$x^\lambda$ to $V_{n-1}$ is of Jordan type $(\overline{\lambda})^T$
for some $\overline{\lambda} \in \Lambda^+(n-1,d-\mu_n)$ 
such that $\lambda_{i+1} \leq \overline{\lambda}_i \leq \lambda_{i}$
for all $1 \leq i \leq n-1$.
By induction we deduce for 
$j=0,1,\dots,n$ that
the restriction of $x^\lambda$ to $V_j$ is of Jordan type
$(\lambda^{(j)})^T$ for
$\lambda^{(j)} \in \Lambda^+(j,\mu_1+\cdots+\mu_j)$
satisfying $\lambda^{(j+1)}_{i+1} \leq \lambda^{(j)}_i \leq \lambda^{(j+1)}_i$
for all $1 \leq i \leq j$.
It follows that there is a well-defined map
\begin{equation}
J:X^\lambda_\mu \rightarrow \Std^\lambda_\mu,\qquad
(V_0,\dots,V_n) \mapsto \mathtt{S},
\end{equation}
where $\mathtt{S}$ is the semi-standard tableau having entry $j$ in all 
boxes belonging to the Young diagram
of $\lambda^{(j)}$ but not of $\lambda^{(j-1)}$ for
$j=1,\dots,n$.

\begin{Theorem}\label{cpts}
For $\mathtt{S} \in \Std^\lambda_\mu$,
we have that
$J^{-1}(\mathtt{S})$ is a locally closed, smooth, irreducible
subvariety of $X^\lambda_\mu$ of dimension $d_\lambda-d_\mu$.
Moreover 
$Y_{\mathtt{S}}^\circ$ is a dense open subset of $J^{-1}(\mathtt{S})$.
\end{Theorem}

\begin{proof}
For $\mathtt{T} \in \Col^\lambda_\mu$, define $\mathtt{T}^+ \in \Std^\lambda_\mu$
as follows.
Let $\overline{\mathtt{T}} \in \Col^{\overline{\lambda}}_{\overline{\mu}}$
be as
in Definition~\ref{combdef}. 
Then let $\mathtt{T}^+$ be obtained from the recursively
defined $(\overline{\mathtt{T}})^+ \in \Std^{\overline{\lambda}}_{\overline{\mu}}$
by adding the entry $n$ to all boxes
belonging to the Young diagram of $\lambda$ but not of $\overline{\lambda}$.
All points 
of $Y_{\mathtt{T}}^\circ$
map to $\mathtt{T}^+$ under the map $J$. Hence
for $\mathtt{S} \in \Std^\lambda_\mu$ we have that
\begin{equation}\label{step}
J^{-1}(\mathtt{S}) = \bigcup_{\mathtt{T} \in \Omega(\mathtt{S})} Y_{\mathtt{T}}^\circ
\quad\text{ where }
\quad\Omega(\mathtt{S}) := \{\mathtt{T} \in \Col^\lambda_\mu\:|\:\mathtt{T}^+ 
= \mathtt{S}\}.
\end{equation}
In particular, $J^{-1}(\mathtt{S})$ is locally closed as each $Y_{\mathtt{T}}^\circ$ is so by Theorem~\ref{pavingt}.
Note further that 
$\mathtt{S}$ belongs to $\Omega(\mathtt{S})$,
and all other $\mathtt{T} \in \Omega(\mathtt{S})$ 
are strictly smaller than $\mathtt{S}$ in the partial order $\preceq$, 
so Theorem~\ref{pavingt}
implies
$Y_{\mathtt{S}}^\circ$ is open in $J^{-1}(\mathtt{S})$.

It remains to prove that $J^{-1}(\mathtt{S})$ is smooth
and irreducible of the given dimension.
Define $\gamma$ and
$\overline{\mathtt{S}} \in \Std^{\overline{\lambda}}_{\overline{\mu}}$
as in Definition~\ref{combdef}, using $\mathtt{S}$ in place of $\mathtt{T}$. 
Let $(c_1,\dots,c_k)$ be the column sequence of 
$\gamma$. In other words, the $c_i$'s index
the columns of $\mathtt{S}$ containing the 
entry $n$.
Also define $\bar\pi$ as in Proposition~\ref{pavingl}(i).
Let
$\Omega_0(\mathtt{S}):=\{\mathtt{T} \in \Omega(\mathtt{S})\:|\:
\text{$\mathtt{T}$ 
has entry $n$ in columns $c_1,\dots,c_k$}\}$, 
so that
\begin{equation}\label{step2}
M := 
J^{-1}(\mathtt{S}) \cap \bar\pi^{-1}(Y_\gamma^\circ)
 =
\bigcup_{\mathtt{T} \in \Omega_0(\mathtt{S})} Y_{\mathtt{T}}^\circ.
\end{equation}
Observe that 
every element of $\Omega(\mathtt{S}) \setminus \Omega_0(\mathtt{S})$ 
is smaller than
every element of $\Omega_0(\mathtt{S})$ in the partial order $\preceq$.
Hence using Theorem~\ref{pavingt} again, we deduce that
$M$
is an open subset of $J^{-1}(\mathtt{S})$.
The map $\mathtt{T} \mapsto \overline{\mathtt{T}}$ is a bijection
between $\Omega_0(\mathtt{S})$ and $\Omega(\overline{\mathtt{S}})$.
So, comparing (\ref{step2}) with
(\ref{step}) for $\overline{\mathtt{S}}$,
we see that
the isomorphism $\bar f_\gamma$ from Proposition~\ref{pavingl}(iv)
restricts to an isomorphism
$M \stackrel{\sim}{\rightarrow} Y_\gamma^\circ \times
J^{-1}(\overline{\mathtt{S}})$.
By induction, we deduce that 
$M$ is a smooth irreducible variety of dimension
$|\gamma|+d_{\overline{\lambda}} - d_{\overline{\mu}}
= d_\lambda-d_\mu$.

Finally let $Z$ be the centralizer of $x^\lambda$ in $GL_d(\mathtt{C})$.
This is a connected algebraic group which acts naturally on
$X^\lambda_\mu$. Moreover $Z$ leaves the subset $J^{-1}(\mathtt{S})$
invariant, so the action restricts to a morphism
$$
m:Z \times M \rightarrow J^{-1}(\mathtt{S}).
$$
To complete the proof of the theorem, we just need to show that this
map is surjective.
Take $(V_0,\dots,V_n) \in 
J^{-1}(\mathtt{S})$.
By (\ref{step}), 
it lies in $Y_{\mathtt{T}}^\circ$ for some $\mathtt{T} \in \Omega(\mathtt{S})$.
Suppose that the entries of $\mathtt{T}$ equal to $n$ are in columns
$1 \leq b_1 < \cdots < b_k \leq s$.
There is a permutation of the columns of equal height in 
the Young diagram of $\lambda$ sending columns $b_1,\dots,b_k$ to
columns $c_1,\dots,c_k$. Let $x \in GL_d(\C)$ be the
matrix inducing the associated
permutation of the basis
$f_1,\dots,f_d$, labelling boxes as in (\ref{piccy}).
So we have that 
$x (\langle f_{b_1},\dots,f_{b_k}\rangle^\perp) = \langle f_{c_1},\dots,f_{c_k}\rangle^\perp$,
and obviously $x \in Z$.
Also, as in  the proof of Proposition~\ref{pavingl}, there is an element 
$y := g(V_{n-1}^\perp)^T \in Z$ such that $y (V_{n-1}) = 
\langle f_{b_1},\dots,f_{b_k} \rangle^\perp$.
Hence $xy(V_{n-1}) = \langle f_{c_1},\dots,f_{c_k}\rangle^\perp$,
so $xy \in Z$ maps $(V_0,\dots,V_n)$ to a point of $M$.
This implies that $m$ is surjective.
\end{proof}

\begin{Corollary}[Spaltenstein]
Assuming $\lambda \geq \mu^+$, $X^\lambda_\mu$ is equidimensional of dimension $d_\lambda-d_\mu$, with irreducible components
$Y_{\mathtt{S}} 
:= 
\overline{Y_{\mathtt{S}}^\circ}=\overline{J^{-1}(\mathtt{S})}$ 
for 
$\mathtt{S} \in \Std^\lambda_\mu$
(closure
in the Zariski topology).
\end{Corollary}

\begin{proof}
By Theorem~\ref{cpts}, the subvarieties
$J^{-1}(\mathtt{S})$ for  $\mathtt{S} \in \Std^\lambda_\mu$ are
irreducible of dimension $d_\lambda-d_\mu$,
and they partition $X^\lambda_\mu$ into disjoint subsets. Hence their closures give all the irreducible
components.
Finally, $Y_{\mathtt{S}}^\circ$ has the same closure as $J^{-1}(\mathtt{S})$ 
 as it is a dense subset.
\end{proof}

\begin{Corollary}
For $\mathtt{T} \in \Col^\lambda_\mu$,
we have that $\deg(\mathtt{T}) \leq d_\lambda-d_\mu$,
with equality if and only if $\mathtt{T}$ is semi-standard.
\end{Corollary}

\begin{proof}
Let $\mathtt{S} := \mathtt{T}^+$, notation as in the proof of Theorem~\ref{cpts}.
By Theorem~\ref{cpts} and the decomposition (\ref{step}),
$Y_{\mathtt{T}}^\circ$ 
is an irreducible subset of the irreducible variety
$J^{-1}(\mathtt{S})$, and it is dense 
in $J^{-1}(\mathtt{S})$
if and only if $\mathtt{T}$
is semi-standard (equivalently, $\mathtt{T} = \mathtt{S}$).
The corollary follows from this 
since $\dim J^{-1}(\mathtt{S}) = d_\lambda-d_\mu$
and $\dim Y_{\mathtt{T}}^\circ = \deg(\mathtt{T})$.
(It is not hard to supply a purely combinatorial proof of this corollary.)
\end{proof}

\section{Algebraic basis}\label{sab}

Continue with fixed $\la\in \La^+(n,d)$ and $\mu \in \La(n,d)$.
Recall the elements $h_r(\mu;i_1,\dots,i_m) \in P_\mu$ 
and the algebra $C^\lambda_\mu := P_\mu / I^\lambda_\mu$ 
from the introduction. 

\begin{Lemma}\label{nz}
The algebra $C^\lambda_\mu$ is non-zero if and only if
$\lambda \geq \mu^+$.
\end{Lemma}

\begin{proof}
Since everything is graded, $C^\la_\mu$ is non-zero
if and only if all the generators from (\ref{rel1}) are of positive
degree, i.e. $\lambda_1+\cdots+\lambda_m \geq \mu_{i_1}+\cdots+\mu_{i_m}$
for all $m \geq 1$ and $1 \leq i_1<\cdots<i_m\leq n$.
By the definition of the dominance ordering on partitions, this
is the statement that $\lambda \geq \mu^+$.
\end{proof}

For $\mathtt{T} \in \Col^\lambda_\mu$, we inductively define
an element $h(\mathtt{T}) \in P_\mu$ as follows.
If $n = d = 0$ then $\mathtt{T}$ is the empty tableau
and we simply set $h(\mathtt{T}) := 1$. If $n \geq 1$, we
let 
$\gamma \in \Par_{k}$ and 
$\overline{\mathtt{T}} \in 
\Col^{\overline{\lambda}}_{\overline{\mu}}$ be as in 
Definition~\ref{combdef}.
The natural embedding
$\C[x_1,\dots,x_{d-k}] \hookrightarrow
\C[x_1,\dots,x_d]$ induces an embedding
$P_{\overline{\mu}} \hookrightarrow P_\mu$.
This allows us to view the
recursively defined element $h(\overline{\mathtt{T}}) \in 
P_{\overline{\mu}}$ as an element of $P_\mu$. Then we set
\begin{equation}
h(\mathtt{T}) := h(\overline{\mathtt{T}}) h_\gamma(\mu;n)
\qquad\text{ where }\qquad
h_\gamma(\mu;n) := \prod_{i=1}^k h_{\gamma_i}(\mu;n).
\end{equation}
Recalling (\ref{degdef}), $h(\mathtt{T})$ is homogeneous of degree
$2\deg(\mathtt{T})$.
For example, if $\mathtt{T}$ is as in Example~\ref{anex}
then $h(\mathtt{T}) = h_1(\mu;3) h_1(\mu;6) = x_6(x_{11}+x_{12})$.
We use the same notation for the canonical image of 
$h(\mathtt{T})$ in the quotient $C^\lambda_\mu$.

\begin{Theorem}\label{bt}
The elements $\{h(\mathtt{T})\:|\:\mathtt{T}\in\Col^\lambda_\mu\}$
give a basis for $C^\lambda_\mu$.
\end{Theorem}

In remainder of the section, we will prove the spanning part of 
Theorem~\ref{bt},
postponing the proof of 
linear independence to $\S$\ref{stan}.
The approach is similar to \cite[Lemmas 2--3]{T} where the case
of regular $\mu$ was treated, but the general case
turns out to be considerably more delicate.

The argument goes by induction on $n$.
The theorem is trivial in the case $n=0$,
so assume for the rest of the section that $n\geq 1$
and that we have proved the spanning part of Theorem~\ref{bt} for all
smaller $n$.
 Let $k := \mu_n$, $s := \lambda_1$ and $t := \lambda_n$.
In view of Lemma~\ref{nz}, we may as well assume 
that $\lambda \geq \mu^+$, hence we have that
$s \geq k \geq t$.
Set $\beta := ((s-k)^{k-t}) \in \Par_k$, and let
$\overline{\mu} \in \Lambda(n-1,d-k)$ be obtained from $\mu$
by forgetting the last part (cf. Proposition~\ref{pavingl}).
Let $\unrhd$ be the partial order on the set $\Par_k$ of partitions of height at most $k$ such
such that $\gamma \unrhd \kappa$ if either $|\gamma| > |\kappa|$, or
$|\gamma| = |\kappa|$ and
$\gamma \geq \kappa$ in the dominance ordering.
For $\gamma \in \Par_k$,
let $J_{\unrhd \gamma}$ (resp.\ $J_{\rhd \gamma}$) denote the
ideal of $P_\mu$ generated by $I^\lambda_\mu$
and
all $h_\kappa(\mu;n)$ for $\kappa \unrhd \gamma$ (resp.\ $\kappa \rhd \gamma$).

\begin{Lemma}\label{hard}
Fix $m \geq 0$, $1 \leq i_1 < \cdots < i_m \leq n-1$,
$r \geq 0$ and $\gamma \in \Par_k$.
Suppose we are given $1 \leq c \leq k$ such that
\begin{align}\label{eq1}
r+c &> \lambda_1+\cdots+\lambda_m-\mu_{i_1}-\cdots-\mu_{i_m},\\
r+\gamma_c &> \lambda_1+\cdots+\lambda_m+\lambda_{m+1}-\mu_{i_1}-\cdots-\mu_{i_m} - \mu_n.\label{eq2}
\end{align}
Then
$h_r(\mu;i_1,\dots,i_m) h_\gamma(\mu;n) \in J_{\rhd \gamma}$.
\end{Lemma}

\begin{proof}
We first formulate and prove two technical claims.
By a {\em marked partition} we mean a pair $(\pi;p)$ consisting
of a partition $\pi$ and a non-zero part $p$ of $\pi$.
We write $\pi \cup \{q\}$ for the partition obtained by adding
one extra part equal to $q$ to the partition $\pi$.

\vspace{2mm}
\noindent{\bf Claim 1.}
{\em Suppose we are given vectors $v(\pi;p)$ 
for each marked partition $(\pi;p)$ with $1 \leq |\pi| \leq c$.
For any partition $\pi$ with $|\pi| \leq c$, 
let $v(\pi)$ denote $\sum_p v(\pi;p)$ 
summing over the set of all non-zero parts $p$ of $\pi$.
Assume for $1 \leq b \leq c$ that
$v(\pi) +
\sum_{q=1}^b v(\pi\cup\{q\};q) = 0$
for each partition $\pi$ of $(c-b)$.
Then we have that
$\sum_{|\pi|=c}
(-1)^{h(\pi)} v(\pi) = 0.$}

\vspace{1mm}
\noindent
To see this, we note for $i=1,\dots,c$ that
\begin{align*}
\sum_{\substack{ |\pi|  < c \\ h(\pi) = i-1}} (-1)^{h(\pi)}
v(\pi)&=
\sum_{\substack{|\pi| < c \\ h(\pi) = i-1}}
\sum_{q=1}^{c-|\pi|} (-1)^{h(\pi)+1} v(\pi\cup\{q\};q)\\
&=
\sum_{\substack{|\pi| < c \\ h(\pi) = i}} (-1)^{h(\pi)}
v(\pi) +
\sum_{\substack{|\pi| = c \\ h(\pi) = i}} (-1)^{h(\pi)}
v(\pi).
\end{align*}
Hence,
\begin{align*}
v(\varnothing) 
&= 
\sum_{\substack{|\pi| < c \\ h(\pi) = 0}}
(-1)^{h(\pi)} v(\pi) + 
\sum_{\substack{|\pi| = c \\ h(\pi) \leq 0}}
(-1)^{h(\pi)} v(\pi)\\
&= 
\sum_{\substack{|\pi| < c \\ h(\pi) = 1}}
(-1)^{h(\pi)} v(\pi)+\sum_{\substack{|\pi| = c \\ h(\pi) \leq 1}}
(-1)^{h(\pi)} v(\pi)\\
&\qquad\qquad\qquad\qquad\quad\:\qquad\vdots\\
&=
\sum_{\substack{|\pi| < c \\ h(\pi) = c}}
(-1)^{h(\pi)} v(\pi)+ 
\sum_{\substack{|\pi| = c \\ h(\pi) \leq c}}
(-1)^{h(\pi)} v(\pi)
=
\sum_{|\pi| = c}
(-1)^{h(\pi)} v(\pi).
\end{align*}
Since $v(\varnothing) = 0$ this establishes Claim 1.

\vspace{2mm}

Call a function $f:\{1,\dots,c\} \rightarrow \{1,\dots,c\}$
a {\em quasi-permutation} of {\em descent} $b \geq 0$
if there exist distinct integers $1 \leq j_1,\dots,j_{b+1} \leq c$
and a bijection $\bar f:\{1,\dots,c\}\setminus\{j_1,\dots,j_b\}
\rightarrow \{1,\dots,c\}\setminus\{j_1,\dots,j_b\}$ such that
\begin{itemize}
\item $j_1 = c$;
\item $f(j_1) = j_2, f(j_2)=j_3,\dots,f(j_b) = j_{b+1}$;
\item $f(j) = \bar f(j)$ for each $j \in \{1,\dots,c\}\setminus\{j_1,\dots,j_b\}$.
\end{itemize}
Note $b$, $j_1,\dots,j_{b+1}$ and $\bar f$ are uniquely determined by
the quasi-permutation $f$, and
quasi-permutations of descent $0$ are ordinary
permutations belonging to the symmetric group $S_c$.
The {\em marked cycle type} of
the quasi-permutation $f$ is the marked partition
$(\pi;p)$
defined by letting
$\pi$ be the partition of $(c-b)$ 
recording the usual cycle type of the permutation
$\bar f$ of $\{1,\dots,c\} \setminus \{j_1,\dots,j_b\}$
and $p$ be the length of the cycle that involves $j_{b+1}$ when
$\bar f$ is written as a product of disjoint cycles.
For example, the quasi-permutation  
$$
\left(
\begin{array}{lllllllllllllll}
1&2&3&4&5&6&7&8&9&10&11&12&13&14&15\\
2&1&5&6&7&10&8&9&7&4&12&13&14&11&3
\end{array}
\right)
$$
of $c=15$ has marked cycle type $(4\,3^2\,2;3)$.
Let $S_c(\pi;p)$ denote the set of all quasi-permutations
of marked cycle type $(\pi;p)$.
In particular, $S_c((1^c);1) = \{\operatorname{id}\}$.
Given $f \in S_c(\pi;p)$, let
\begin{align*}
\omega(f) &:= (1-f(1),2-f(2),\dots,c-f(c)) \in \Z^c,\\
|\omega(f)| &:= (1-f(1))+\cdots+(c-f(c)) \in \Z_{\geq 0},\\
h_{\gamma-\omega(f)}(\mu;n) &:= 
\prod_{a=1}^c
h_{\gamma_a-a+f(a)}(\mu;n)
\prod_{a =c+1}^{k} h_{\gamma_a}(\mu;n).
\end{align*}

\vspace{2mm}
\noindent{\bf Claim 2.}
{\em 
For each marked partition $(\pi;p)$ with $1 \leq |\pi| \leq c$, we set
$$
v(\pi;p) :=
\sum_{f \in S_c(\pi;p)} h_{r+|\omega(f)|}(\mu;i_1,\dots,i_m)
h_{\gamma-\omega(f)}(\mu;n).
$$
Also define $v(\pi)$ be as in Claim~1.
Then, for $1 \leq b \leq c$ and each partition $\pi$ of $(c-b)$, we have that
$v(\pi) + \sum_{q=1}^b v(\pi\cup\{q\};q) \in J_{\rhd \gamma}$.}

\vspace{1mm}
\noindent
To establish this, fix $1 \leq b \leq c$ and 
$\pi$ as in the claim and let 
$$
\Omega := \left\{(\bj;g)\:\Bigg|\:
\begin{array}{l}
\bj = (j_1,\dots,j_b)\text{ a tuple of distinct integers with}\\
j_1=c\text{ and } 1 \leq j_2,\dots,j_b < c;\:\:\:
g\text{ a permutation}\\ 
\text{of }\{1,\dots,c\}\setminus\{j_1,\cdots,j_b\}
\text{ of cycle type $\pi$}
\end{array}
\right\}.
$$
Given $(\bj;g) \in \Omega$ and $1 \leq i \leq c$, we let
$f_i(\bj;g):\{1,\dots,c\} \rightarrow \{1,\dots,c\}$ be the 
function mapping $j_1 \mapsto j_2,\dots,j_{b-1} \mapsto j_b,
j_b \mapsto i$ and $j \mapsto g(j)$ for all
$j \in \{1,\dots,c\}\setminus \{j_1,\dots,j_b\}$.
If $i \in \{1,\dots,c\}\setminus\{j_1,\dots,j_b\}$
then $f_i(\bj;g)$ is a quasi-permutation of descent $b$
and marked cycle type $(\pi;p)$ where $p$ is the length of the
cycle of $g$ containing $i$.
Otherwise, we have that $i = j_{b-q+1}$ for some $1 \leq q \leq b$,
and $f_i(\bj;g)$ is 
a quasi-permutation of descent $(b-q)$
and marked cycle type $(\pi\cup\{q\};q)$.
It follows that
$|\omega(f_i(\bj;g))| = c-i$ and 
$$
v(\pi) + \sum_{q=1}^b v(\pi\cup\{q\};q) = 
\sum_{(\bj;g) \in \Omega}
\sum_{i=1}^c h_{r+c-i}(\mu;i_1,\dots,i_m)
h_{\gamma-\omega(f_i(\bj;g))}(\mu;n).
$$
Therefore it suffices to show for each $(\bj;g) \in \Omega$ that
$$
x:=\sum_{i=1}^c h_{r+c-i}(\mu;i_1,\dots,i_m)
h_{\gamma-\omega(f_i(\bj;g))}(\mu;n)
$$
belongs to $J_{\rhd \gamma}$.
Let
$$
y := 
\prod_{\substack{a=1 \\ a \neq j_b}}^c
h_{\gamma_a - a+f_i(\bj;g)(a)}(\mu;n)
\prod_{a =c+1}^k h_{\gamma_a}(\mu;n)
$$
so that
$$
x
=
\sum_{i=1}^c h_{r+c-i}(\mu;i_1,\dots,i_m)
h_{\gamma_{j_b} - j_b+i}(\mu;n)y.
$$
Using (\ref{easyexp}), we can expand
$$
h_{r+c+\gamma_{j_b}-j_b}(\mu;i_1,\dots,i_m,n)
=
\sum_{i=j_b-\gamma_{j_b}}
^{r+c} h_{r+c-i}(\mu;i_1,\dots,i_m) h_{\gamma_{j_b-j_b+i}}(\mu;n).
$$
As $j_b \leq c$, we have that $c + \gamma_{j_b}-j_b \geq \gamma_c$, hence
using (\ref{eq2}) and (\ref{rel1})
we get that
$h_{r+c+\gamma_{j_b}-j_b}(\mu;i_1,\dots,i_m,n) \in I^\lambda_\mu$.
Also using (\ref{eq1}) 
we have that $h_{r+c-i}(\mu;i_1,\dots,i_m) \in I^\lambda_\mu$ 
for $i \leq 0$.
We deduce that
$$
\sum_{i=1}^{r+c} h_{r+c-i}(\mu;i_1,\dots,i_m) h_{\gamma_{j_b-j_b+i}}(\mu;n)\in I^\lambda_\mu.
$$
As $I^\lambda_\mu \subseteq J_{\rhd \gamma}$,
this implies that
$$
x \equiv - \sum_{i = c+1}^{c+r} 
h_{r+c -i}(\mu;i_1,\dots,i_m) h_{\gamma_{j_b}-j_b+i}(\mu;n)y \pmod{J_{\rhd \gamma}}.
$$
But for $i > c$ the term $h_{\gamma_{j_b}-j_b+i}(\mu;n) y$
is of the form $h_\kappa(\mu;n)$ for some $\kappa \in \Par_k$ 
with $|\kappa| > |\gamma|$,
so it belongs to $J_{\rhd \gamma}$.
This proves Claim 2.

\vspace{2mm}
Now we can complete the proof of the lemma.
Let $v(\pi;p)$ be as in Claim 2 and then define $v(\pi)$ as in Claim 1.
From Claims 1 and 2, we get that
$$
\sum_{|\pi|=c} (-1)^{h(\pi)-c} v(\pi) \in J_{\rhd \gamma}.
$$
For $|\pi| = c$, the set $S_c(\pi;p)$ 
appearing in the definition of $v(\pi;p)$
consists of quasi-permutations of 
descent zero, i.e. ordinary permutations.
So the above sum is equal to
$v((1^c)) = 
h_r(\mu;i_1,\dots,i_m)h_\gamma(\mu;n)$
plus a linear combination of terms 
$h_{r+|\omega(f)|}(\mu;i_1,\dots,i_m) h_{\gamma-\omega(f)}(\mu;n)$
for ordinary permutations $f \neq \operatorname{id}$.
It remains to observe for all such $f$ that
$\omega(f) < 0$ in the dominance ordering on $\Z^c$,
hence
$h_{\gamma-\omega(f)}(\mu;n) \in J_{\rhd \gamma}$.
\end{proof}

\begin{Lemma}\label{T1} For $\gamma \in \Par_k$ with
$\gamma \not\subseteq\beta$, we have that $J_{\unrhd\gamma} = 
J_{\rhd\gamma}$. Moreover $J_{\unrhd\gamma} = I^\lambda_\mu$ if $|\gamma| > k(s-k)$.
\end{Lemma}

\begin{proof}
If $\gamma_1 > s-k=\la_1-\mu_n$ then
$h_{\gamma_1}(\mu;n) \in I^\lambda_\mu$ by (\ref{rel1}),
hence $h_\gamma(\mu;n) \in I^\lambda_\mu \subseteq J_{\rhd \gamma}$.
In particular, this shows $J_{\unrhd \gamma} = I^\lambda_\mu$
when $|\gamma| > k(s-k)$, as in that case all
$\gamma \unlhd \kappa \in \Par_k$ satisfy $\kappa_1 > s-k$.
It remains to show that
$h_\gamma(\mu;n) \in J_{\rhd \gamma}$
for $\gamma \not\subseteq\beta$ with $\gamma_1 \leq s-k$.
In that case, we have that $t \geq 1$ and $\gamma_{k-t+1} > 0$.
Now apply Lemma~\ref{hard}, taking $r=0,
m=n-1$ and $c = k-t+1$ and noting the right hand sides of 
(\ref{eq1}) and (\ref{eq2}) equal $k-t$ and $0$, respectively.
\end{proof}

\begin{Lemma}\label{T2}
Take
$\gamma \in \Par_{k}$ 
such that $\gamma \subseteq\beta$.
Define $(c_1,\dots,c_k)$ and 
$\overline{\la} \in \Lambda^+(n-1,d-k)$ as in 
the statement of Proposition~\ref{pavingl}. 
The quotient $J_{\unrhd \gamma} / J_{\rhd \gamma}$
is spanned by the images of the elements $h(\mathtt{T})$
for $\mathtt{T} \in \Col^\lambda_\mu$ such that
$\mathtt{T}$ has entry $n$ in each of columns $c_1,\dots,c_k$.
\end{Lemma}

\begin{proof}
Consider the homomorphism
$\C[x_1,\dots,x_d] \rightarrow
\C[x_1,\dots,x_{d-k}]$ mapping $x_i\mapsto x_i$ for $1 \leq i \leq d-k$ and $x_i\mapsto 0$ for $d-k+1 \leq i \leq d$.
It restricts to a homomorphism
$P_\mu \rightarrow P_{\overline{\mu}}$
such that $h_r(\mu;i) \mapsto h_r(\overline{\mu};i)$
for $1 \leq i \leq n-1$,
and $h_r(\mu;n) \mapsto 0$ for $r \geq 1$.
Let $\Phi:P_\mu \rightarrow C^{\overline{\lambda}}_{\overline{\mu}}
$ be the composite of this homomorphism
with the natural quotient map $P_{\overline{\mu}} \rightarrow C^{\overline{\lambda}}_{\overline{\mu}}$.

The $P_\mu$-module $J_{\unrhd \gamma} / J_{\rhd \gamma}$
is cyclic, generated by the image of $h_\gamma(\mu;n)$.
We claim that the action of $P_\mu$ on 
$J_{\unrhd \gamma} / J_{\rhd \gamma}$ factors through $\Phi$ to make $J_{\unrhd \gamma} / J_{\rhd \gamma}$ into a well-defined
cyclic $C^{\overline{\la}}_{\overline{\mu}}$-module.
To prove this, we need to show that
$(\ker \Phi) J_{\unrhd \gamma} \subseteq J_{\rhd \gamma}$.
By definition, $\ker\Phi$ is generated by 
the elements $\{h_r(\mu;n)\:|\:r > 0\}$
together with the elements
\begin{equation*}
\left\{
h_r(\mu;i_1,\dots,i_m)\:\bigg|\: 
\begin{array}{l}
m \geq 1, 
1 \leq i_1 < \cdots < i_m \leq n-1,\\
r > \overline{\lambda}_1+\cdots+\overline{\lambda}_m-\mu_{i_1}-\cdots-\mu_{i_m}
\end{array}
\right\}.
\end{equation*}
So we need to show that multiplication by either
of these families of elements
sends $h_\gamma(\mu;n)$ into $J_{\rhd \gamma}$.

The first family is easy to deal with: for $r > 0$ the element
$h_r(\mu;n) h_\gamma(\mu;n)$ is a symmetric polynomial in the variables
$x_{d-k+1},\dots,x_d$
of degree stricly greater than $|\gamma|$. Hence it can be expressed
as a linear combination of $h_\kappa(\mu;n)$'s
with $|\kappa| > |\gamma|$, as the $h_\kappa(\mu;n)$'s
for $\kappa \in \Par_k$
give a basis
for the space of all symmetric polynomials
in $x_{d-k+1},\dots,x_d$.

To deal with the second family, we need to show that
$$
h_r(\mu;i_1,\dots,i_m)h_\gamma(\mu;n) \in J_{\rhd \gamma}
$$
for $m \geq 1$, $1 \leq i_1 < \cdots < i_m \leq n-1$
and $r > \overline{\lambda}_1+\cdots+\overline{\lambda}_m-\mu_{i_1}-\cdots-\mu_{i_m}$.
Let $c := \lambda_1+\cdots+\lambda_m - \overline{\lambda}_1-\cdots-\overline{\lambda}_m$.
If $c = 0$ then $h_r(\mu;i_1,\dots,i_m) \in I^\lambda_\mu$ already,
so there is nothing to do.
So we may assume that $c \geq 1$, and are in the situation of Lemma~\ref{hard}.
The hypothesis (\ref{eq1}) is immediate, so we are left with checking
(\ref{eq2}). To see that, note that
$\overline{\lambda}$ has $c$ fewer boxes than $\lambda$ on the first $m$
rows. So by the definition of $\overline{\lambda}$,
we must have that $k+1-c+\gamma_c > \lambda_{m+1}$.
Hence
\begin{align*}
r + \gamma_c &>
\overline{\lambda}_1+\cdots+\overline{\lambda}_m-\mu_{i_1}-\cdots-\mu_{i_m}
+\lambda_{m+1}-k\\
&\geq
\lambda_1+\cdots+\lambda_m+\lambda_{m+1} -\mu_{i_1}-\cdots-\mu_{i_m}-\mu_n,
\end{align*}
and the claim is proved.

Now to prove the lemma, we have shown that $J_{\unrhd \gamma} / J_{\rhd \gamma}$ is a cyclic
$C^{\overline{\lambda}}_{\overline{\mu}}$-module generated by the image of
$h_\gamma(\mu;n)$. By the induction hypothesis being used to prove the spanning part of Theorem~\ref{bt}, we know that
$C^{\overline{\lambda}}_{\overline{\mu}}$
is spanned by the elements $h(\overline{\mathtt{T}})$ for
$\overline{\mathtt{T}} \in \Col^{\overline{\lambda}}_{\overline{\mu}}$.
Hence $J_{\gamma}$ is spanned by the images of the elements
$h(\mathtt{T}) = h(\overline{\mathtt{T}}) h_\gamma(\mu;n)$ as described in the 
statement of the lemma.
\end{proof}

Now we can complete the proof of the spanning part of Theorem~\ref{bt}.
Enumerate the partitions $\gamma \subseteq\beta$
as 
$\gamma^{(1)}=\varnothing, 
\gamma^{(2)}, \dots, \gamma^{(N-1)},\gamma^{(N)} = \beta$
so that $\gamma^{(i)} \unlhd \gamma^{(j)}$ implies $i \leq j$.
Let $J_i$ be the ideal of $C^\lambda_\mu$
generated by 
$\{h_{\gamma^{(j)}}(\mu;n)\:|\:i \leq j \leq N\}$. Then
$$
C^\lambda_\mu = J_1 \geq \cdots\geq J_N \geq J_{N+1} := \{0\}
$$
is a filtration of $C^\lambda_\mu$.
Lemma~\ref{T1} 
implies for $i=1,\dots,N$ that the 
canonical map
$J_{\unrhd\gamma^{(i)}} \rightarrow J_i$ induces a surjection
$J_{\unrhd\gamma^{(i)}} / J_{\rhd\gamma^{(i)}} \twoheadrightarrow J_i / J_{i+1}$.
Hence by Lemma~\ref{T2} we get that
$J_i / J_{i+1}$ is spanned by the images of the $h(\mathtt{T})$ for
$\mathtt{T} \in \Col^\lambda_\mu$ with
entry $n$ in each of columns $c_1,\dots,c_k$, where 
$(c_1,\dots,c_k)$ is the column sequence of $\gamma^{(i)}$.
Hence the $h(\mathtt{T})$ for all $\mathtt{T} \in \Col^\lambda_\mu$
span $C^\lambda_\mu$ itself.

\section{Proof of Theorem~\ref{mt2}}\label{sbm}

In order to prove Theorem~\ref{mt2},
we need to exploit the 
construction
of the Springer representations via 
perverse sheaves. There are two basic
approaches, one (by restriction) due
to Lusztig, Borho and Macpherson \cite{BM1},
and the other (by Fourier transform) due to Kashiwara and Brylinski
\cite{Brylinski}. We refer to
\cite[ch.13]{Jantzen} and \cite[$\S$6]{Gi2} for more recent accounts
of these two approaches,
and also \cite[$\S$VI.15]{KW} which explains in detail the relationship between 
them.
Some of these references work 
in terms of \'etale cohomology
so require some translation before they can be used in our complex setting.

For any complex variety $Y$,
we let $D^b(Y)$ be the bounded derived category of constructible
sheaves of $\C$-vector spaces on $Y$; see e.g. \cite[$\S$3]{Gi2}.
Let $$
\mathbf{D}:D^b(Y) \rightarrow D^b(Y)
$$ be the 
Verdier duality functor.
Let $\Perv(Y)$ be the (abelian) full subcategory of $D^b(Y)$ 
consisting of perverse sheaves; see e.g. \cite[$\S$4]{Gi2}.
The constant sheaf on $Y$ is denoted
$\mathbb{C}_Y$, which we often
view as an object in $D^b(Y)$ concentrated in degree zero.
For $\mathcal M \in D^b(Y)$, $\mathcal M[i]$ denotes the 
object obtained from $\mathcal M$ by translating down by $i$.
If $V = \bigoplus_{j \in \Z} V_j$ is a graded vector space,
we'll write $\mathcal M \otimes V$ for
$\bigoplus_{j \in \Z} \mathcal M[-j] \otimes V_j$,
so $\mathcal M \otimes (V[i]) = (\mathcal M[i]) \otimes V 
= (\mathcal M \otimes V)[i]$.

Assume from now on that $Y$ is irreducible and smooth.
In that case, $\C_Y[\dim Y]$ is a perverse sheaf.
For a holomorphic vector bundle $E \rightarrow Y$,
let $D^b_{mon}(E)$ be the derived category of the category
of bounded complexes of sheaves of $\C$-vector spaces
on $E$ whose cohomology sheaves are monodromic, i.e. locally constant
over orbits of the natural $\C^\times$-action on $E$.
Let $\Perv_{mon}(E)$ be the full subcategory of $\Perv(E)$ consisting
of the monodromic perverse sheaves.
We will need the Fourier transform
$$
\mathbf{F}:D^b_{mon}(E) \rightarrow D^b_{mon}(E^*),
$$
where $E^* \rightarrow Y$ is the dual bundle;
see
\cite[$\S$8]{Gi2} or
\cite[$\S$6]{Brylinski}
for its definition (our $\mathbf{F}$ is the normalized 
Fourier transform denoted $\widetilde{\mathcal F}$ on \cite[p.69]{Brylinski}).
The Fourier transform induces an equivalence of categories
$$
\mathbf{F}:\Perv_{mon}(E) \rightarrow \Perv_{mon}(E^*)
$$
(see \cite[Corollaire 7.23]{Brylinski}), which
corresponds under the Riemann-Hilbert correspondence
to the formal Fourier transform 
on holonomic $D$-modules with regular singularities
(see \cite[Th\'eor\`eme 7.24]{Brylinski}).

\begin{Lemma}\label{vbl}
Let $E \rightarrow Y$ be a vector bundle on the smooth irreducible
variety $Y$ as above.
Let $\iota:V\hookrightarrow E$ be a sub-bundle with annihilator $\bar\iota:V^\circ \hookrightarrow E^*$.
Also let $\hat\iota:\bz\hookrightarrow E$ be the zero sub-bundle.
The unit of adjunction
$\Id \rightarrow \bar \iota_* \circ \bar \iota^{-1}$ defines a
canonical map $\res:\C_{E^*}\rightarrow \bar\iota_* \C_{V^\circ}$.
Similarly there is a canonical map
$\res:\iota_* \C_V
\rightarrow \hat\iota_* \C_{\sbz}$
defined by applying 
$\iota_*$ to the unit of adjunction for the inclusion $\bz \hookrightarrow V$.
There are unique (up to scalars) horizontal isomorphisms
in the following diagram:
$$
\begin{CD}
\mathbf{DF}(\hat\iota_* \C_{\sbz}[ \dim V])&@>\sim>>&\C_{E^*}[\dim V^\circ]
\\
@V\mathbf{DF}(\res) VV&&@VV \res V\\
\mathbf{DF}(\iota_* \C_V[\dim V])&@>\sim>>&
\bar\iota_* \C_{V^\circ}[\dim V^\circ]
\end{CD}
$$
Moreover the scalars can be chosen so the diagram commutes.
\end{Lemma}

\begin{proof}
As $\bar\iota_* \C_{V^\circ}[\dim V^\circ]$ is self-dual,
the existence of the bottom isomorphism amounts to 
the assertion that $\mathbf{F}(\iota_*\C_V[\dim V]) \cong \bar\iota_*
\C_{V^\circ}[\dim V^\circ]$, which is a basic property of Fourier
transform; see \cite[Proposition 8.3(4)]{Gi2} or 
\cite[Corollary III.13.4]{KW} in the \'etale setting. The uniqueness of the bottom isomorphism follows as $\End(\bar\iota_* \C_{V^\circ}) \cong \C$.
Existence and uniqueness of the top isomorphism is proved in a similar
way.
Finally the commutativity of the diagram is justified in \cite[Remark III.13.6]{KW} (we have applied $\mathbf{D}$ to the statement there).
\end{proof}

Now let $\mathfrak{g} := \mathfrak{gl}_d(\C)$ and
$\mathfrak{b}$ be the Borel subalgebra of upper triangular matrices.
As usual 
$X$ and $X_\mu$ denote the varieties of full flags
and partial flags of type $\mu$ in $\C^d$.
We abbreviate
$$
r := 2\dim X = d(d-1),
\qquad
r_\mu := 2\dim X_\mu = d(d-1)-2d_\mu.
$$
We will always identify $\mathfrak{g}$ with $\mathfrak{g}^*$ via the trace form,
 noting that $\mathfrak{b}^\perp$ is the nilpotent radical $\mathfrak{n}$
of $\mathfrak{b}$.
We can view $\mathfrak{g}$ as a self-dual 
vector bundle over a point, so that the Fourier transform for
$\mathfrak{g}$ gives a self-equivalence
$$
\mathbf{F}:\Perv_{mon}(\mathfrak{g}) \rightarrow \Perv_{mon}(\mathfrak{g}).
$$
We also work with the trivial vector bundles
$\mathfrak{g}\times X \rightarrow X$
and $\mathfrak{g} \times X_\mu \rightarrow X_\mu$. These are again identified
with their duals, so Fourier transform
gives two more self-equivalences
\begin{align*}
\mathbf{F}:\Perv_{mon}(\mathfrak{g}\times X) &\rightarrow \Perv_{mon}(
\mathfrak{g}\times X),\\
\mathbf{F}:\Perv_{mon}(\mathfrak{g}\times X_\mu) &\rightarrow \Perv_{mon}(
\mathfrak{g}\times X_\mu).
\end{align*}

Let $\mathcal N$ be the nilpotent cone in $\mathfrak{g}$, and set
\begin{align*}
\widetilde{\mathcal{N}} &:= 
\{(x,(U_0,\dots,U_d)) \in \mathcal{N} \times X\:|\:x U_i \subseteq U_{i-1}\text{ for each }i=1,\dots,d\},\\
\widetilde{\mathfrak{g}} &:= 
\{(x,(U_0,\dots,U_d)) \in \mathfrak{g} \times X\:|\:x U_i \subseteq U_i\text{ for each }i=1,\dots,d\}.
\end{align*}
These are sub-bundles of $\mathfrak{g}\times X \rightarrow X$,
with fibers isomorphic to $\mathfrak{n}$ and $\mathfrak{b}$, respectively.
In fact, 
$\widetilde{\mathfrak{g}}$ is the annihilator of $\widetilde{\mathcal{N}}$ in 
the self-dual bundle $\mathfrak{g}\times X$, while
$\widetilde{\mathcal{N}}$ is canonically isomorphic
to the cotangent bundle $T^*X$ of the flag variety;
see e.g. \cite[$\S$6.5]{Jantzen}.
Hence both varieties are smooth and irreducible,
with 
$\dim \widetilde{\mathcal N}=2\dim X = r$
and $\dim \widetilde{\mathfrak{g}} = \dim \mathfrak{g} = d^2$.
Analogously in the parabolic case, we consider
\begin{align*}
\widetilde{\mathcal{N}}_\mu
&:= \{(x,(V_0,\dots,V_n)) \in \mathcal{N}\times X_\mu\:|\:
x V_i \subseteq V_{i-1}\text{ for each }i=1,\dots,n\},\\
\widetilde{\mathfrak{g}}_\mu
&:= \{(x,(V_0,\dots,V_n)) \in \mathfrak{g}\times X_\mu\:|\:
x V_i \subseteq V_i\text{ for each }i=1,\dots,n\},
\end{align*}
which are sub-bundles of 
$\mathfrak{g}\times X_\mu \rightarrow X_\mu$ such that
$\widetilde{\mathfrak{g}}_\mu$ is the annihilator 
of $\widetilde{\mathcal N_\mu}$, and
$\widetilde{\mathcal{N}}_\mu$ is isomorphic to the cotangent bundle
$T^* X_\mu$.
So
$\widetilde{\mathcal{N}}_\mu$ and $\widetilde{\mathfrak{g}}_\mu$ are 
smooth, irreducible varieties 
with
$\dim \widetilde{\mathcal{N}}_\mu = 2\dim X_\mu = r_\mu$
and $\dim \widetilde{\mathfrak{g}}_\mu = \dim \mathfrak{g} = d^2$, respectively.

We will often exploit the following inclusions of sub-bundles:
$$
\begin{picture}(160,70)
\put(0,30){$\mathfrak{g}\times X \stackrel{\iota}{\hookleftarrow} \widetilde{\mathfrak{g}}
\hookleftarrow \widetilde{\mathcal N} \hookleftarrow \{0\}\times X$}
\put(42,56){$_{\bar\iota}$}
\put(67,9){$_{\hat\iota}$}
\put(8.69,43){$\downarrow$}
\put(8.66,18){$\uparrow$}
\put(76.5,44){\line(0,1){6}}
\put(122.4,16){\line(0,1){6}}
\put(44,50){\oval(65,26)[t]}
\put(67,16){\oval(111,26)[b]}
\end{picture}
\begin{picture}(170,70)
\put(20,30){$\mathfrak{g}\times X_\mu \stackrel{{\iota}'}{\hookleftarrow} \widetilde{\mathfrak{g}}_\mu
\hookleftarrow \widetilde{\mathcal N}_\mu \hookleftarrow \{0\}\times X_\mu$}
\put(67,56){$_{{\bar\iota}'}$}
\put(95,9){$_{{\hat\iota}'}$}
\put(28.69,43){$\downarrow$}
\put(28.66,18){$\uparrow$}
\put(106.6,44){\line(0,1){6}}
\put(156.6,16){\line(0,1){6}}
\put(69,50){\oval(75,26)[t]}
\put(94,16){\oval(125,26)[b]}
\end{picture}
$$
Composing the named maps with the first
projections $\pr:\mathfrak{g}\times X \rightarrow \mathfrak{g}$ 
and $\pr':\mathfrak{g}\times X_\mu \rightarrow \mathfrak{g}$ gives
the named maps in the following diagrams:
$$
\begin{picture}(140,70)
\put(0,30){$\mathfrak{g} \stackrel{\pi}{\leftarrow} \widetilde{\mathfrak{g}}
\hookleftarrow \widetilde{\mathcal N} \hookleftarrow \{0\}\times X$}
\put(26,56){$_{\bar\pi}$}
\put(50,9){$_{\hat\pi}$}
\put(0.2,43){$\downarrow$}
\put(0.2,18){$\uparrow$}
\put(53,44){\line(0,1){6}}
\put(98,16){\line(0,1){6}}
\put(28,50){\oval(50,26)[t]}
\put(50.5,16){\oval(95,26)[b]}
\end{picture}\qquad
\begin{picture}(140,70)
\put(20,30){$\mathfrak{g} \stackrel{{\pi}'}{\leftarrow} \widetilde{\mathfrak{g}}_\mu
\hookleftarrow \widetilde{\mathcal N}_\mu \hookleftarrow \{0\}\times X_\mu$}
\put(48,56){$_{{\bar\pi}'}$}
\put(74,9){$_{{\hat\pi}'}$}
\put(21.7,43){$\downarrow$}
\put(21.7,18){$\uparrow$}
\put(77.5,44){\line(0,1){6}}
\put(126.6,16){\line(0,1){6}}
\put(51,50){\oval(53,26)[t]}
\put(75.5,16){\oval(102,26)[b]}
\end{picture}
$$
The morphism $\pi:\widetilde{\mathfrak{g}}\twoheadrightarrow \mathfrak{g}$ here is part of {Grothendieck's simultaneous resolution}
of the quotient of $\mathfrak{g}$ by the adjoint action of $G := GL_d(\C)$.
It is a projective, hence proper, morphism
which is small in the sense of Goresky-Macpherson (see e.g.
\cite[Lemma 13.2]{Jantzen}).
Similarly
the map
$\pi':\widetilde{\mathfrak{g}}_\mu
\twoheadrightarrow \mathfrak{g}$
in the parabolic case is small.
Hence $R \pi_* \C_{\widetilde{\mathfrak{g}}}[d^2]$
and
$R\pi'_* \C_{\widetilde{\mathfrak{g}}_\mu}[d^2]$
are (monodromic) perverse sheaves on $\mathfrak{g}$, which we call the
{\em Grothendieck sheaf} and the {\em partial Grothendieck sheaf},
respectively.

There is a natural action of the
Weyl group $S_d$ on the Grothendieck sheaf.
To recall some of the details of the construction of this action,
let $\mathfrak{g}^{rs}$ denote the set of regular semisimple elements in
$\mathfrak{g}$, and set $\widetilde{\mathfrak{g}}^{rs} := \pi^{-1}(\mathfrak{g}^{rs})$. 
Let $\pi^{rs}$ be the restriction of $\pi$
to $\widetilde{\mathfrak{g}}^{rs}$.
Then $\pi^{rs}:\widetilde{\mathfrak{g}}^{rs} \twoheadrightarrow 
\mathfrak{g}^{rs}$ is a regular covering whose automorphism group
is canonically identified with the
Weyl group $S_d$; see e.g. \cite[Proposition 9.3]{Gi2}.
So $S_d$ acts freely on $\widetilde{\mathfrak{g}}^{rs}$ by deck transformations,
and there is a unique isomorphism
between $\mathfrak{g}^{rs}$ and the quotient
$S_d \backslash \widetilde{\mathfrak{g}}^{rs}$ making the following diagram
commute:
\begin{equation}\label{tue}
\begin{CD}
&&\widetilde{\mathfrak{g}}^{rs}\!\!\\
&^{\pi^{rs}}\!\!\swarrow\!\!\!&&\!\!\!\searrow^{\!\operatorname{can}}\\
\mathfrak{g}^{rs}\!\!\!\!\!&@>\sim>>&\!\!\!\!\!\!\!S_d \backslash \widetilde{\mathfrak{g}}^{rs}.
\end{CD}
\end{equation}
The 
action of $S_d$ on $\widetilde{\mathfrak{g}}^{rs}$ induces an action
on the local system 
$\pi^{rs}_* \C_{\widetilde{\mathfrak{g}}^{rs}}$. 
Let $IC$ be the intersection cohomology
functor from local systems
on $\mathfrak{g}^{rs}$ to perverse sheaves on $\mathfrak{g}$.
Applying it to 
$\pi^{rs}_* \C_{\widetilde{\mathfrak{g}}^{rs}}$,
we get an induced
action of $S_d$ on the perverse sheaf $IC(\pi^{rs}_* \C_{\widetilde{\mathfrak{g}}^{rs}})$.
Finally, as in \cite[Corollary 9.2]{Gi2} or \cite[Theorem 13.5]{Jantzen}, 
the smallness of $\pi$ implies that there is a unique isomorphism
\begin{equation}\label{shiso}
\kappa:IC(\pi^{rs}_* \C_{\widetilde{\mathfrak{g}}^{rs}}) \stackrel{\sim}{\rightarrow}
R \pi_* \C_{\widetilde{\mathfrak{g}}}[d^2]
\end{equation}
such that the composition
\begin{equation*}
\pi^{rs}_* \C_{\widetilde{\mathfrak{g}}^{rs}} \stackrel{\sim}{\rightarrow}
\mathcal H^{-d^2} 
(IC(\pi^{rs}_* \C_{\widetilde{\mathfrak{g}}^{rs}}))|_{\mathfrak{g}^{rs}}
\stackrel{\sim}{\rightarrow}
\mathcal H^0 
 (R\pi_* \C_{\widetilde{\mathfrak{g}}})|_{\mathfrak{g}^{rs}}
\stackrel{\sim}{\rightarrow} \pi^{rs}_* \C_{\widetilde{\mathfrak{g}}^{rs}}
\end{equation*}
is the identity, where the first 
and last isomorphisms are canonical
and the middle one is induced by $\kappa$.
Using the isomorphism $\kappa$, we transport the action of $S_d$
to get the desired action on 
$R \pi_* \C_{\widetilde{\mathfrak{g}}}[d^2]$.
It is important to note that
this action
induces an algebra isomorphism
\begin{equation}\label{medford2}
\C S_d \stackrel{\sim}{\rightarrow} \End(R \pi_* \C_{\widetilde{\mathfrak{g}}}[d^2]).
\end{equation}
This follows by the perverse continuation principle as the action
on
$\pi^{rs}_* \C_{\widetilde{\mathfrak{g}}^{rs}}$
obviously defines an isomorphism
$\C S_d \stackrel{\sim}{\rightarrow} \End(\pi^{rs}_*
\C_{\widetilde{\mathfrak{g}}^{rs}})$.

The partial Grothendieck sheaf 
was studied in detail by Borho and Macpherson, who explained how
to identify it with
$S_\mu$-invariants in the Grothendieck sheaf; see especially \cite[Proposition
2.7(a)]{BM}. We need to reformulate this 
result slightly. Consider the map $\tilde{p}:\widetilde{\mathfrak{g}} \rightarrow
\widetilde{\mathfrak{g}}_\mu$
arising from the restriction of the map $\id \times p: \mathfrak{g}\times X
\rightarrow \mathfrak{g} \times X_\mu$.
Applying $R \pi'_*$ to the unit of adjunction
$\operatorname{Id} \rightarrow \tilde p_* \circ \tilde p^{-1}$
evaluated at $\C_{\widetilde{g}_\mu}$, we get a map
\begin{equation}\label{bmmap}
BM:R \pi'_*\C_{\widetilde{\mathfrak{g}}_\mu}[d^2] \rightarrow 
R \pi_* \C_{\widetilde{\mathfrak{g}}}[d^2].
\end{equation}
Note also as
$\pi^{-1}(\{0\}) = \{0\}\times X \cong X$ and 
${\pi'}^{-1}(\{0\}) =\{0\}\times X_\mu \cong X_\mu$ that the proper base change
theorem gives us canonical isomorphisms
$H^*(X\cfs) \cong 
(R \pi_* \C_{\widetilde{\mathfrak{g}}})_0$ 
and $H^*(X_\mu\cfs) \cong
(R \pi_* \C_{\widetilde{\mathfrak{g}}})_0$ 
between $H^*(X\cfs)$ and $H^*(X_\mu\cfs)$ (viewed here as a complexes of vector
spaces rather than graded vector spaces)
and the stalks of the Grothendieck and partial Grothendieck
sheaves
at the origin. 
Moreover, under the isomorphism
$H^*(X\cfs) \cong 
(R \pi_* \C_{\widetilde{\mathfrak{g}}})_0$,
the classical action
of $S_d$ on $H^*(X\cfs)$ 
agrees with the action of $S_d$ on
$(R \pi_* \C_{\widetilde{\mathfrak{g}}})_0$ 
induced by its action on the Grothendieck sheaf; see e.g. \cite[Lemma 13.6]{Jantzen}.

\begin{Theorem}\label{bm}
The morphism $BM$ just defined gives an isomorphism of perverse sheaves
$$
BM:R \pi'_* \C_{\widetilde{\mathfrak{g}}_\mu}[d^2]
\stackrel{\sim}{\rightarrow}
R \pi_* \C_{\widetilde{\mathfrak{g}}}[d^2]^{S_\mu\inv}.
$$
Moreover, letting $BM_0$ denote the
induced map between stalks at the origin,
the following diagram commutes:
\begin{equation}\label{sta}
\begin{CD}
H^*(X_\mu,\C)&@>p^*>>&H^*(X,\C)\\
@V\operatorname{can} VV&&@VV \operatorname{can} V\\
(R \pi'_* \C_{\widetilde{\mathfrak{g}}_\mu})_0
&@>>BM_0>&
(R \pi_* \C_{\widetilde{\mathfrak{g}}})_0
\end{CD}
\end{equation}
\end{Theorem}

\begin{proof}
The second statement is clear from the definition (\ref{bmmap}) and
naturality of proper base change.

For the first statement,
the idea is to compare our morphism $BM$
with the isomorphism
$R \pi'_* \C_{\widetilde{\mathfrak{g}}_\mu}[d^2]
\stackrel{\sim}{\rightarrow}
R \pi_* \C_{\widetilde{\mathfrak{g}}}[d^2]^{S_\mu\inv}$  constructed in \cite[Proposition
2.7(a)]{BM}.
Let $S_\mu \backslash \widetilde{\mathfrak{g}}^{rs}$ denote the quotient of
$\widetilde{\mathfrak{g}}^{rs}$ by the action of the parabolic subgroup $S_\mu$.
Let $\widetilde{\mathfrak{g}}^{rs}_{\mu} := 
{\pi'}^{-1}(\mathfrak{g}^{rs})$ and
${\pi'}^{\,rs}$ be the restriction of $\pi'$
to $\mathfrak{g}^{rs}$.
The key point is that there is a unique isomorphism
$\widetilde{g}^{rs}_\mu \stackrel{\sim}{\rightarrow}
S_\mu \backslash \widetilde{\mathfrak{g}}^{rs}$ making the following diagram
\begin{equation*}
\begin{CD}
&&\widetilde{\mathfrak{g}}^{rs}\!\\
&\!^{\tilde p^{rs}}\!\!\!\swarrow\!\!\!&&\!\!\!\searrow^{\!\operatorname{can}}\\
\widetilde{\mathfrak{g}}^{rs}_\mu\! &@>\sim>>& \!\!\!\!\!\!\!S_\mu \backslash \widetilde{\mathfrak{g}}^{rs}\\
@V{\pi'}^{\,rs} VV&&@VV\operatorname{can}V\\
\mathfrak{g}^{rs}\!&@>\sim>>&\!\!\!\!\!\!\!S_d \backslash \widetilde{\mathfrak{g}}^{rs}
\end{CD}
\end{equation*}
commute, where the bottom isomorphism comes from (\ref{tue})
and $\tilde p^{rs}$ is the restriction of $\tilde p$.
This isomorphism induces an isomorphism of local systems
${\pi'_*}^{rs} \C_{\widetilde{\mathfrak{g}}^{rs}_\mu}
\stackrel{\sim}{\rightarrow}
\left(\pi^{rs}_* \C_{\widetilde{\mathfrak{g}}^{rs}}\right)^{S_\mu\inv}$.
Now apply the functor $IC$ 
to get an isomorphism of perverse sheaves
$$
IC({\pi'_*}^{rs} \C_{\widetilde{\mathfrak{g}}^{rs}_\mu})
\stackrel{\sim}{\rightarrow}
IC(\pi^{rs}_* \C_{\widetilde{\mathfrak{g}}^{rs}})^{S_\mu\inv}.
$$
The Borho-Macpherson isomorphism 
$R \pi'_* \C_{\widetilde{\mathfrak{g}}_\mu}[d^2]
\stackrel{\sim}{\rightarrow}
R \pi_* \C_{\widetilde{\mathfrak{g}}}[d^2]^{S_\mu\inv}$ is obtained
from this by composing on the right with the inverse of the isomorphism $\kappa$
from (\ref{shiso}) and on the left with the analogously defined isomorphism 
$\kappa':IC({\pi_*'}^{rs} \C_{\widetilde{\mathfrak{g}}_\mu^{rs}}) \stackrel{\sim}{\rightarrow}
R \pi'_* \C_{\widetilde{\mathfrak{g}}_\mu}[d^2]$
from \cite[Proposition 2.6]{BM}.

To complete the proof it remains to observe that our map $BM$ is
equal to the map from the previous paragraph. This follows by the perverse
continuation principle, since the two maps agree on restriction
to $\mathfrak{g}^{rs}$.
\end{proof}

As in the statement of Lemma~\ref{vbl}, the units of adjunction 
associated to the maps $\bar\iota:\widetilde{\mathcal N} \hookrightarrow
\mathfrak{g} \times X$ and $\hat\iota:\{0\}\times X \hookrightarrow \mathfrak{g}\times X$
induce canonical
maps $\res:\C_{\mathfrak{g}\times X} \rightarrow \bar \iota_* \C_{\widetilde{\mathcal N}}$
and
$\res:\hat\iota_* \C_{\{0\}\times X} \rightarrow
\iota_* \C_{\widetilde{\mathfrak{g}}}$.
Applying the derived functor $R\, \pr_*$,
we get morphisms
\begin{align}\label{alph}
\alpha:
\C_{\mathfrak{g}} \otimes H^*(X\cfs)&\rightarrow
R \bar\pi_* \C_{\widetilde{\mathcal{N}}},\\
\beta:
R \pi_* \C_{\widetilde{\mathfrak{g}}}
&\rightarrow 
\delta_0 \otimes H^*(X\cfs),\label{bet}
\end{align} 
where we are
viewing $H^*(X\cfs)$ as a graded vector space with 
$H^i(X,\C)$ in degree $i$, 
$\delta_0$
denotes the constant skyscraper sheaf on $\mathfrak{g}$
supported at the origin, and we have implicitly composed with the
canonical isomorphisms
$R\,\pr_* \C_{\mathfrak{g}\times X} \cong \C_{\mathfrak{g}} \otimes H^*(X\cfs)$, 
$R\,\pr_*(\bar\iota_* \C_{\widetilde{\mathcal N}}) \cong R \bar\pi_* \C_{\widetilde{\mathcal N}}$,
$R\,\pr_*(\hat\iota_*\C_{\{0\}\times X})
\cong R \hat\pi_* \C_{\{0\}\times X} \cong \delta_0 \otimes H^*(X\cfs)$
 and $R\,\pr_*(\iota_* \C_{\widetilde{\mathfrak{g}}}) \cong R \pi_* \C_{\widetilde{\mathfrak{g}}}$.
Also let
\begin{equation}
\gamma:\mathbf{DF}(\delta_0)\otimes H_*(X\cfs)\rightarrow
\mathbf{DF}(R \pi_* \C_{\widetilde{\mathfrak{g}}})
\end{equation}
be the morphism obtained from $\beta$ by applying the composite functor $\mathbf{DF}$,
noting that $\mathbf{DF}(\delta_0 \otimes H^*(X\cfs))
\cong \mathbf{DF}(\delta_0) \otimes H_*(X\cfs)$ if
$H_*(X\cfs)$
is identified with the graded dual of the graded vector space $H^*(X\cfs)$ (so
$H_i(X\cfs)$ is in degree $-i$).
Repeating all this in exactly the same way in the parabolic case, we get analogous morphisms
\begin{align}
\alpha':
\C_{\mathfrak{g}} \otimes H^*(X_\mu\cfs)&\rightarrow
R \bar\pi'_* \C_{\widetilde{\mathcal{N}}_\mu},\\
\beta':
R \pi'_* \C_{\widetilde{\mathfrak{g}}_\mu}
&\rightarrow \delta_0 \otimes H^*(X_\mu\cfs),\\
\gamma':\mathbf{DF}(\delta_0)\otimes H_*(X_\mu\cfs)&\rightarrow
\mathbf{DF}(R \pi'_* \C_{\widetilde{\mathfrak{g}}_\mu}).
\end{align}
Finally note 
that there exists a unique (up to scalars) isomorphism
\begin{equation}\label{sigma}
\sigma: \mathbf{DF}(\delta_0[d^2]) \stackrel{\sim}{\rightarrow} \C_{\mathfrak{g}},
\end{equation}
as follows for example by a special case of Lemma~\ref{vbl}.
We fix a choice of such a map.
In the regular case the following proposition is essentially 
\cite[Theorem VI.15.1]{KW}. We are repeating some of the details of the proof
to make it clear that it extends to the parabolic case.

\begin{Proposition}\label{taus}
There exist unique
isomorphisms $\tau$ and $\tau'$ making the following diagrams commute:
\begin{equation}\label{face1}
\begin{CD}
\mathbf{DF}(\delta_0[d^2]) \otimes H_*(X,\C)
&@>\sim>\sigma \otimes PD>&
\C_{\mathfrak{g}} \otimes H^*(X,\C)[r]\\@V\gamma VV&&@VV\alpha V\\
\mathbf{DF}
(R \pi_* \C_{\widetilde{\mathfrak{g}}}[d^2])
&@>\sim>\tau >&
R \bar\pi_* \C_{\widetilde{\mathcal N}}[r]
\end{CD}\:\:\:\:\:\:\:
\end{equation}
\begin{equation}\label{face2}
\begin{CD}
\mathbf{DF}(\delta_0[d^2])\otimes H_*(X_\mu,\C)
&@>\sim>\sigma\otimes PD'>&
\C_{\mathfrak{g}} \otimes H^*(X_\mu,\C) [r_\mu]\\
@V\gamma' VV&&@VV\alpha' V\\
\mathbf{DF}(R \pi'_* \C_{\widetilde{\mathfrak{g}}_\mu}[d^2])
&@>\sim>\tau' >&
R \bar\pi'_* \C_{\widetilde{\mathcal N}_\mu}[r_\mu]
\end{CD}\:\:\:\:\:
\end{equation}
Here $PD:H_{i}(X\cfs) \rightarrow H^{r-i}(X\cfs)$
and
$PD':H_{i}(X_\mu\cfs) \rightarrow H^{r_\mu-i}(X_\mu\cfs)$
are the isomorphisms defined by Poincar\'e duality.
\end{Proposition}

\begin{proof}
We just sketch the proof for $\tau$, since the argument for $\tau'$ is
similar. 

For existence, we apply Lemma~\ref{vbl}
with $E = \mathfrak{g}\times X = E^*$, $V = \widetilde{\mathfrak{g}}$ and
$V^\circ = \widetilde{\mathcal N}$ to get a commuting diagram
$$
\begin{CD}
\mathbf{DF}(\hat\iota_* \C_{\{0\}\times X}[d^2])&@>\sim>>&\C_{\mathfrak{g}\times X}[r]
\\
@V\mathbf{DF}(\res) VV&&@VV \res V\\
\mathbf{DF}(\iota_* \C_{\widetilde{\mathfrak{g}}}[d^2])&@>\sim>>&
\bar\iota_* \C_{\widetilde{\mathcal N}}[r]
\end{CD}
$$
Now apply $R\,\pr_*$, using Verdier duality and the 
fact that Fourier transform commutes with direct images (see 
\cite[Proposition 6.8]{Brylinski} or \cite[Theorem III.13.3]{KW} in the \'etale setting) which
together give  us an isomorphism of functors $\mathbf{DF} \circ R\, \pr_*
\stackrel{\sim}{\rightarrow}
R\, \pr_* \circ \mathbf{DF}$. We get the following commutative diagram:
$$
\begin{CD}
\mathbf{DF}(R\,\pr_* (\hat\iota_* \C_{\{0\}\times X}[d^2]))&\:\stackrel{\sim}{\longrightarrow}\:
&R\,\pr_*(\mathbf{DF}(\hat\iota_* \C_{\{0\}\times X}[d^2]))&\:\stackrel{\sim}{\longrightarrow}\:&R\,\pr_*(\C_{\mathfrak{g}\times X}[r])
\\
@V\mathbf{DF}(R\,\pr_*(\res)) VV@VR\,\pr_*(\mathbf{DF}(\res)) VV@V R\,\pr_*(\res) VV\\
\mathbf{DF}(R\,\pr_* (\iota_*
\C_{\widetilde{\mathfrak{g}}}[d^2]))&\:\stackrel{\sim}{\longrightarrow}\:& 
R\,\pr_*(\mathbf{DF}(\iota_* \C_{\widetilde{\mathfrak{g}}}[d^2]))&\:\stackrel{\sim}{\longrightarrow}\:&
R \,\pr_*(\bar\iota_* \C_{\widetilde{\mathcal N}}[r])
\end{CD}
$$
Recalling the identifications made in the definitions of $\alpha$, $\beta$ and
$\gamma$ above, this gives the commuting square (\ref{face1}) on checking that the top map in the diagram
constructed is equal up to
a scalar to the map $\sigma \otimes PD$ at the top of (\ref{face1}).
The latter statement amounts to showing that the graded
vector space maps $\zeta:H^*(X\cfs) \rightarrow H^*(X\cfs)$ and
$\xi:H_*(X\cfs) \rightarrow H^*(X\cfs)[r]$ defined by the following
commutative diagrams are equal (up to scalars) to the maps $\id$ and
$PD$, respectively
(the vertical identifications in these diagrams are various maps induced by unique up to
scalars isomorphisms):
$$
\begin{CD}
\mathbf{F}(R\, \pr_*(\hat\iota_* \C_{\{0\}\times X}[d^2]))&@>\sim>>&R\,\pr_*(\mathbf{F}(\hat\iota_* \C_{\{0\}\times X}[d^2]))\\
@|&&@|\\
\C_\mathfrak{g} \otimes H^*(X\cfs) [2d^2]&@>\id \otimes \zeta>>&\C_{\mathfrak{g}}\otimes H^*(X\cfs) [2d^2]
\end{CD}
$$
$$
\begin{CD}
\mathbf{D}(R\, \pr_* (\C_{\mathfrak{g}\times
  X}[2d^2]))&@>\sim>>&R\,\pr_*(\mathbf{D} 
(\C_{\mathfrak{g}\times X}[2d^2]))\\
@|&&@|\\
\C_\mathfrak{g} \otimes H_*(X\cfs)&@>\id \otimes \xi>>&\C_{\mathfrak{g}}\otimes H^*(X\cfs)[r]
\end{CD}
$$
To see that $\xi$ is proportional to $PD$, push-forward to a point
and use the usual connection between Verdier duality and Poincar\'e
duality.
It is more difficult to see that $\zeta$ is proportional to $\id$ because the
isomorphism
$R\,\pr_* \circ \mathbf{F} \cong \mathbf{F} \circ R\,\pr_*$ is only
defined implicitly. One way to avoid this issue is to
apply the functor $\rho^!$ where
$\rho:\{0\}\hookrightarrow \mathfrak{g}$ is the inclusion, 
noting that the induced isomorphism
$\rho^! \circ R\,\pr_* \circ \mathbf{F} \cong \rho^! \circ \mathbf{F} \circ
R\,\pr_*$ of triangulated functors
is actually unique (up to scalars)
as both functors are isomorphic to the global hypercohomology functor
${\mathbf H}^*(\mathfrak{g}\times X,?)$.

For the uniqueness of $\tau$, it suffices to show that the map
$$
\End\left(\mathbf{DF}(R \pi_* \C_{\widetilde{\mathfrak{g}}}[d^2])\right)
\rightarrow
\hom\left(\mathbf{DF}(\delta_0[d^2]) \otimes H_*(X,\C), \mathbf{DF}(R \pi_* \C_{\widetilde{\mathfrak{g}}}[d^2])\right)
$$
mapping $f$ to $f \circ \gamma$
is injective.
Equivalently, the map
$$
\End\left(R \pi_* \C_{\widetilde{\mathfrak{g}}}[d^2]\right)
\rightarrow
\hom\left(R
  \pi_* \C_{\widetilde{\mathfrak{g}}}[d^2],\delta_0[d^2] \otimes H^*(X,\C)\right)
$$
mapping $f$ to $\beta \circ f$ is injective. To see the latter
statement, suppose that 
$\beta \circ f = 0$. Then the induced map $(\beta
\circ f)_0 = \beta_0 \circ f_0$ between stalks at $0$ is also
zero.
But the stalks at $0$ of all the spaces under consideration are
naturally identified with $H^*(X,\C)$ in such a way 
that $\beta_0 = \operatorname{id}$, hence we get that $f_0 = 0$.
It remains to observe that the map
$$
\End(R \pi_* \C_{\widetilde{\mathfrak{g}}})
\rightarrow
\End(H^*(X\cfs)), \qquad 
f \mapsto f_0
$$
is injective.
This follows from (\ref{medford2}), 
recalling that $H^*(X\cfs)$
is isomorphic to the group algebra
$\C S_d$ as
a module over the symmetric group thanks to a classical
result of Chevalley \cite{Ch} about the coinvariant algebra.
(The analogous statement needed in the parabolic case, namely, that
the map
$$
\End(
R \pi'_* \C_{\widetilde{\mathfrak{g}}_\mu})
\rightarrow
\End(H^*(X_\mu\cfs)),
\qquad f \mapsto f_0
$$
is injective, is a consequence of the injectivity
just established
 in the regular case, since
$R \pi'_* \C_{\widetilde{\mathfrak{g}}_\mu}$ is isomorphic to a summand of
$R \pi_* \C_{\widetilde{\mathfrak{g}}}$ by
Theorem~\ref{bm} and the semisimplicity of $\C S_d$.)
\end{proof}

In view of Proposition~\ref{taus},
$R \bar \pi_* \C_{\widetilde{\mathcal{N}}}[r]$
and
$R \bar \pi'_* \C_{\widetilde{\mathcal{N}}_\mu}[r_\mu]$
are (monodromic) perverse sheaves on $\mathfrak{g}$, which we call the
{\em Springer sheaf} and the {\em Spaltenstein sheaf}, respectively.
Letting $h:\mathcal N \hookrightarrow \mathfrak{g}$ be the inclusion
and $\tilde\pi:\widetilde{\mathcal N} \rightarrow \mathcal N$
be the usual Springer resolution, we have canonical isomorphisms
\begin{equation}\label{pbtm}
R \bar\pi_* \C_{\widetilde{\mathcal N}}[r]
\cong h_*(R \tilde\pi_* \C_{\widetilde{\mathcal{N}}}[r])
\cong h_*(R \pi_* \C_{\widetilde{\mathfrak{g}}}[r]|_{\mathcal N}),
\end{equation}
the first of which follows  as $\bar\pi = h \circ \tilde\pi$,
and the second comes from the proper base change theorem,
noting that
$\widetilde{\mathcal{N}} = \pi^{-1}(\mathcal{N})$.
If we apply the functor $h_*\circ h^{-1}$ to the action of
$S_d$ on the Grothendieck sheaf,
we get an action of $S_d$ on the perverse sheaf on the right hand side
of (\ref{pbtm}). Using the above isomorphisms we lift this to
an
action of $S_d$ on the
Springer sheaf itself.

\begin{Proposition}\label{ab}
The maps $\alpha$ and $\beta$ from (\ref{alph})--(\ref{bet}) 
are $S_d$-equivariant, where the actions of $S_d$ on 
$\C_{\mathfrak{g}} \otimes H^*(X\cfs)$ and
$\delta_0 \otimes H^*(X\cfs)$ are induced by its natural action
on $H^*(X\cfs)$.
\end{Proposition}

\begin{proof}
This follows as in \cite[Proposition VI.14.1]{KW} (our maps $\alpha$ and
$\beta$ are compositions of pairs of the maps considered there).
\end{proof}

The following important corollary is essentially \cite[Corollary VI.15.2]{KW}.

\begin{Corollary}\label{hob}
The isomorphism $\tau:\mathbf{DF}(R \pi_* \C_{\widetilde{\mathfrak{g}}}[d^2]) \stackrel{\sim}{\rightarrow}
R \bar\pi_* \C_{\widetilde{\mathcal N}}[r]$ in (\ref{face1}) is $S_d$-equivariant up to
a twist by sign, i.e. we have that $\tau \circ w = \sgn(w) w \circ
\tau$ for each $w \in S_d$, where the action of $S_d$ on
$\mathbf{DF}(R \pi_* \C_{\widetilde{\mathfrak{g}}}[d^2])$ is defined
by applying the functor $\mathbf{DF}$ to its action on the
Grothendieck sheaf.
\end{Corollary}

\begin{proof}
Proposition~\ref{ab} implies that the vertical maps $\alpha$ and
$\gamma$ in the diagram (\ref{face1}) are $S_d$-equivariant, where the
action on $\mathbf{DF}(\delta_0[d^2]) \otimes H_*(X\cfs)$ arises from
the dual of the natural action on $H^*(X\cfs)$.
Also the map $\sigma \otimes PD$ at the top of the diagram (\ref{face1}) is
$S_d$-equivariant up to a twist by sign because 
the top cohomology
$H^{r}(X,\C)$ is the sign representation of the symmetric group.
The corollary now follows from the commutativity of the diagram and the
uniqueness of $\tau$ in Proposition~\ref{taus}.
\end{proof}

Now we are ready for the main theorem of the section, which roughly
speaking is the Fourier transform of Borho-Macpherson's Theorem~\ref{bm}.
Recall the isomorphisms $\tau$ and $\tau'$ from Proposition~\ref{taus},
and the morphism $BM$ from (\ref{bmmap}).
Let
\begin{equation}
\overline{BM}:= \tau' \circ \mathbf{DF}(BM) \circ \tau^{-1}:
R \bar\pi_* \C_{\widetilde{\mathcal N}}[r]
\rightarrow
R \bar\pi'_* \C_{\widetilde{\mathcal{N}}_\mu}[r_\mu].
\end{equation}
Although $\tau$ and $\tau'$ depend up to a scalar on the choice of the
map
$\sigma$ in (\ref{sigma}), it is clear that the map $\overline{BM}$ is
 independent
of this choice.

\begin{Theorem}\label{mt3}
The morphism $\overline{BM}$ just defined restricts to an isomorphism of perverse
sheaves 
$$
\overline{BM}:
R \bar\pi_* \C_{\widetilde{\mathcal N}}[r]^{S_\mu\anti}
\stackrel{\sim}{\rightarrow}
R \bar\pi'_* \C_{\widetilde{\mathcal{N}}_\mu}[r_\mu].
$$ 
Moreover, the following diagram commutes:
\begin{equation}\label{lastsq}
\begin{CD}
\C_{\mathfrak{g}} \otimes H^*(X,\C) [r]
&@>> \id \otimes p_* >& \C_{\mathfrak{g}}\otimes
H^*(X_\mu,\C) [r_\mu]\\
@V\alpha VV&&@VV\alpha' V\\
R \bar\pi_* \C_{\widetilde{\mathcal{N}}}[r]&@>>\overline{BM} >&
R \bar\pi'_* \C_{\widetilde{\mathcal{N}}_\mu}[r_\mu]
\end{CD}
\end{equation}
\end{Theorem}

\begin{proof}
Consider the following cube:
{\small\begin{diagram} 
\hspace{-18mm}\mathbf{DF}(\delta_0[d^2])\otimes H_*(X,\C)\hspace{-1mm}& & \rTo^{\hspace{-35mm}\scriptstyle\id\otimes p_*\hspace{-38mm}} & & 
\hspace{-1mm}\mathbf{DF}(\delta_0[d^2])\otimes H_*(X_\mu,\C) \hspace{-12mm}& & \\ 
& \rdTo^{\scriptstyle{\sigma\otimes PD}} & & &
\vLine^{\scriptstyle \gamma'} &
\rdTo^{\scriptstyle{\sigma\otimes PD'}} & \\ 
\dTo^{\scriptstyle \gamma} & & \hspace{-10mm}\C_{\mathfrak{g}}\otimes H^*(X,\C)[r] \hspace{-1mm}& \rTo^{\hspace{-15mm}\scriptstyle\id\otimes p_*\hspace{-33mm}} & \HonV & &  
\hspace{-1mm}\C_{\mathfrak{g}}\otimes H^*(X_\mu,\C) [r_\mu] 
\hspace{-10mm}\\ 
& & \dTo^{\scriptstyle\alpha} & & \dTo & & \\ 
\hspace{-18mm}\mathbf{DF}(R \pi_* \C_{\widetilde{\mathfrak{g}}}[d^2]) & \hLine
 & \VonH & \rTo^{\hspace{-5mm}\scriptstyle\mathbf{DF}(BM)}
\hspace{-8mm} & 
\hspace{7mm}\mathbf{DF}(R \pi'_* \C_{\widetilde{\mathfrak{g}}_\mu}[d^2])\hspace{-1mm}& & \dTo_{\scriptstyle\alpha'} \\ 
& \rdTo^{\scriptstyle\tau} & & & & \rdTo^{\scriptstyle\tau'} & \\ 
& & \hspace{-5mm}\R \bar\pi_* \C_{\widetilde{\mathcal{N}}}[r] \hspace{-1mm}& & \rTo^{\hspace{-5mm}\scriptstyle\overline{BM} \hspace{-5mm}} & & \hspace{-1mm}R \bar\pi'_* \C_{\widetilde{\mathcal{N}}_\mu}[r_\mu]
\end{diagram} 
}

\noindent
In order to prove the second statement of the theorem, we need to show
that the front face commutes, which follows if we can show that all
the other faces commute. The commutativity of the top face is clear, the left and right faces
commute by Proposition~\ref{taus},
and the bottom face commutes by the definition of $\overline{BM}$.
Thus it remains to show that the back face commutes.
Equivalently, by the definitions of $\gamma$ and $\gamma'$,
we need to show that the following
diagram commutes:
$$
\begin{CD}
\delta_0 \otimes H^*(X,\C)
&@<< \id \otimes p^* <& \delta_0\otimes
H^*(X_\mu,\C)\\
@A\beta AA&&@AA\beta' A\\
R \pi_* \C_{\widetilde{\mathfrak{g}}}&@<<BM<&
R \pi'_* \C_{\widetilde{\mathfrak{g}}_\mu}
\end{CD}
$$
To see this, we just need to check at the level of stalks.
It is trivial apart from the stalk at $0$, and at $0$ it
follows from (\ref{sta}).

To deduce the first statement, consider the bottom square in our
commuting cube. Applying the functor $\mathbf{DF}$ to the first
statement  of Theorem~\ref{bm}, we know already that
$\mathbf{DF}(BM)$ restricts to an isomorphism between
$\mathbf{DF}(R \pi_* \C_{\widetilde{\mathfrak{g}}}[d^2])^{S_\mu\inv}$
and
$\mathbf{DF}(R \pi'_* \C_{\widetilde{\mathfrak{g}}_\mu}[d^2])$.
Now use Corollary~\ref{hob}.
\end{proof}

\begin{Remark}\rm
Although we are only considering the case $\mathfrak{g} = \mathfrak{gl}_d(\C)$ here, Theorem~\ref{mt3} 
remains true (with appropriate notational changes) on replacing $\mathfrak{g}$
by an arbitrary semisimple Lie algebra.
The proof in general is essentially the same as above.
\end{Remark}

Finally we explain how to deduce Theorem~\ref{mt2} from Theorem~\ref{mt3}.
The uniqueness in the statement of Theorem~\ref{mt2}
is clear as the map $i^*$ is surjective.
To prove the existence of $\bar p_*$,
apply the functor
$\mathcal H^{i-r}(?)_{x^\lambda}$ (the stalk of the $(i-r)$th cohomology sheaf
at the point $x^\lambda$) to the commutative square (\ref{lastsq}),
and note by proper base change that there are canonical isomorphisms
\begin{align}\label{idone}
\mathcal H^{i-r}(R \bar\pi_* \C_{\widetilde{\mathcal N}}[r])_{x^\lambda}
&\cong
H^i(\bar\pi^{-1}(x^\lambda), \C_{\widetilde{\mathcal N}})
\cong H^i(X^\lambda,\C),\\
\mathcal H^{i-r}(R \bar\pi'_* \C_{\widetilde{\mathcal N}_\mu}[r_\mu])_{x^\lambda}
&\cong
H^{i-2d_\mu}({{\bar\pi}'}{^{-1}}(x^\lambda),\C_{\widetilde{\mathcal N}_\mu})
\cong
H^{i-2d_\mu}(X_\mu^\lambda,\C).
\end{align}
Under these isomorphisms, 
the maps induced by $\alpha$ and $\alpha'$ correspond
to the maps $i^*$ and $j^*$ in (\ref{phys}), respectively, as a consequence of naturality in the 
proper base change theorem.
The last statement of Theorem~\ref{mt2} follows from the first statement of Theorem~\ref{mt3}, since by Proposition~\ref{ab} the 
action of $S_d$ 
on $H^i(X^\lambda\cfs)$ induced by
its action on $R \bar\pi_* \C_{\widetilde{\mathcal{N}}}[r]$ via (\ref{idone})
is exactly the Springer action being considered in Theorem~\ref{mt2}.

\section{Proof of Theorems~\ref{mt1} and \ref{bt}}\label{stan}

Now everything is ready to prove Theorem~\ref{mt1}
(and complete the proof of Theorem~\ref{bt}).
Consider the diagram (\ref{diag2}).
In order to show that the composition $j^*\circ \psi$ factors through 
the quotient to induce a homorphism $\bar\psi:C^\lambda \rightarrow 
H^*(X^\lambda\cfs)$, we need to show that
\begin{equation}\label{s1}
j^*(\psi(e_r(\mu;i_1,\dots,i_m))) = 0
\end{equation}
for $m, i_1,\dots,i_m$ and $r$ as in (\ref{rel2}).
Let $\nu \in \Lambda(n,d)$ be a composition obtained by rearranging the parts of $\mu$ so that $\nu_j 
= \mu_{i_j}$ for each $j=1,\dots,m$. Let
$w \in S_d$ be the unique permutation of minimal length such that
$w S_\nu w^{-1} = S_\mu$.
The natural action of $w$ on $P$ induces an algebra isomorphism
$w:C_\nu \stackrel{\sim}{\rightarrow} C_\mu$.

\begin{Lemma}\label{tim}
There exist unique algebra isomorphisms $f$ and $g$ making the
following diagram commute:
$$
\begin{CD}
C_\nu&@>\psi>>&H^*(X_\nu\cfs)&@>j^*>>&H^*(X^\lambda_\nu\cfs)\\
@Vw VV&&@Vf VV&&@VVg V\\
C_\mu&@>\psi>>&H^*(X_\mu\cfs)&@>j^*>>&H^*(X^\lambda_\mu\cfs).
\end{CD}
$$
\end{Lemma}

\begin{proof}
The existence and uniqueness of $f$ making the left hand square commute is clear,
as the other three maps in this square are algebra isomorphisms. For
the right hand square, we 
first observe that the diagram
$$
\begin{CD}
C^{S_\nu\anti}&@>\sim>>&C_\nu\\
@Vw VV&&@VV w V\\
C^{S_\mu\anti}&@>\sim>>&C_\mu
\end{CD}
$$
commutes, where the horizontal maps are as in (\ref{antii}) and the map $w$
is induced by the natural action of $w$ on $C$.
This follows using the algebraic description of the horizontal maps
given just after (\ref{antii}), together with the
observation 
that $w(\eps_\nu) = \eps_\mu$ by the choice of $w$.
Combining this with (\ref{antii}), we deduce that the left face of the following
cube commutes:
{\small
\begin{diagram} 
\!\!\!\!H^*(X\cfs)^{S_\nu\anti} & & \rTo^{\scriptstyle i^*} & & H^*(X^\la\cfs)^{S_{\nu}\anti} & & \\ 
& \rdTo^{\scriptstyle p_*} & & & \vLine^{\scriptstyle w} & \rdTo^{\scriptstyle\bar p_*} & \\ 
\dTo^{\scriptstyle w} & & H^*(X_\nu\cfs) & \rTo^{\scriptstyle\:j^*\!\!\!\!\!\!\!\!\!\!\!\!\!\!\!\!} & \HonV & & H^*(X^\lambda_\nu\cfs)\phantom{.}\\ 
& & \dTo^{\scriptstyle f} & & \dTo & & \\ 
\!\!\!\!H^*(X\cfs)^{S_\mu\anti} & \hLine & \VonH & \rTo^{\scriptstyle\!\!\!\!\!\!\!\!\!\!\!\!\!\!\!\! i^*} & H^*(X^\lambda\cfs)^{S_\mu\anti} & & \dTo_{\scriptstyle g} \\ 
& \rdTo^{\scriptstyle p_*} & & & & \rdTo^{\scriptstyle \bar p_*} & \\ 
& & H^*(X_\mu\cfs) & & \rTo^{\scriptstyle j^*} & & H^*(X^\lambda_\mu\cfs).
\end{diagram} 
}

\noindent
The vertical maps on the back face are induced by the natural action
of $w$ on $H^*(X\cfs)$ and on $H^*(X^\lambda\cfs)$, respectively, and
it is clear that this face commutes 
because 
$i^*:H^*(X\cfs) \twoheadrightarrow H^*(X^\lambda\cfs)$
is $S_d$-equivariant.
Also the top and bottom faces commute by Theorem~\ref{mt2}.
Finally, we let $g$ be the unique isomorphism making the right face
commute. Then the front face commutes too, and this establishes
the existence of the map $g$ as in the statement of the lemma.
The 
uniqueness of $g$ and the fact that it 
is an algebra homomorphism follow 
because $f$ is a homomorphism and
$j^*$ is surjective according to Corollary~\ref{cd}.
\end{proof}

Continuing with the notation fixed just before Lemma~\ref{tim}, we observe that
$w(e_r(\nu;1,\dots,m)) = e_r(\mu;i_1,\dots,i_m)$. So
Lemma~\ref{tim} implies that
$$
j^*(\psi(e_r(\mu;i_1,\dots,i_m)))
= g(j^*(\psi(e_r(\nu;1,\dots,m)))).
$$
Replacing $\mu$ by $\nu$ if necessary,
this reduces the proof of (\ref{s1}) to the special case
that $i_j = j$ for each $j=1,\dots,m$, 
i.e. we just need to show
that
\begin{equation}\label{s3}
j^*(\psi(e_r(x_1,\dots,x_k))) = 0
\end{equation}
for all $r \geq k-h$,
where
$k := \mu_1+\cdots+\mu_m$,
$h := \lambda_{l+1}+\cdots+\lambda_n$ and
$l := \#\{i=m+1,\dots,n\:|\:\mu_i > 0\}$.
In view of Corollary~\ref{ne}, we may as well assume further that
$\lambda \geq \mu^+$, hence that $h \leq k$.

In order to establish (\ref{s3}), 
we follow
Tanisaki's original argument from \cite{T} closely.
We work again with the Schubert varieties $Y_\gamma \subseteq \Gr_{k,d-k}$ from (\ref{schub}). However 
we switch to using the basis 
$f_1,\dots,f_d$ for $\C^d$ obtained by
reading the boxes of the Young diagram in the opposite order to 
$\S$\ref{spaving}, 
for example in the situation of (\ref{piccy}) we now use the labelling
\begin{equation}\label{piccy2}
\Youngdiagram{$f_9$&$f_8$&$f_7$&$f_6$\cr $f_5$&$f_4$&$f_3$\cr $f_2$&$f_1$\cr}
\qquad\qquad
x^\lambda = \begin{CD}
@VVV
\end{CD}.
\end{equation}
The fundamental classes 
$[Y_\gamma] \in H_{2|\gamma|}(\Gr_{k,d}\cfs)$
give a basis for $H_*(\Gr_{k,d}\cfs)$.
Let 
$\{\sigma_\gamma\:|\:\gamma \in \Par_{k,d-k}\}$ be the dual basis for
$H^*(\Gr_{k,d}\cfs)$.

\begin{Lemma}\label{tiger}
Let
$s:X_\mu \rightarrow \Gr_{k,d},\: (V_0,\dots, V_n) \mapsto V_m$
be the natural projection.
We have that
$s^*(\sigma_\gamma) = \psi(s_\gamma(x_1,\dots,x_k))$
for each $\gamma \in \Par_{k,d-k}$,
where $s_\gamma(x_1,\dots,x_k)$ is the image in $C_\mu$ of the
Schur function 
associated to the partition $\gamma$.
\end{Lemma}

\begin{proof}
This follows from the analogous statement for the full flag variety $X$, which is classical. To give a little more detail, consider the projection
$r:X \rightarrow \Gr_{k,d},\: (U_0,\dots,U_d) \mapsto U_k$.
By \cite[$\S$10.6]{Ful}, we have that
$r^*(\sigma_\gamma) = \phi(s_\gamma(x_1,\dots,x_k))$.
Now use (\ref{invi})
and the fact that $r = s \circ p$.
\end{proof}

Set $\alpha := ((d-k)^{h-k}) \in \Par_{k,d-k}$
and consider the Schubert variety 
\begin{equation*}
Y_\alpha = \{U \in \Gr_{k,d}\:|\: \langle f_1,\dots,f_h \rangle \subseteq U\}.
\end{equation*}
It is obvious that $Y_\alpha \cong \Gr_{k-h,d-h}$.

\begin{Lemma}\label{sq2}
The projection $s:X_\mu \rightarrow \Gr_{k,d}$ from Lemma~\ref{tiger}
restricts to a morphism $\bar s:X^\lambda_\mu \rightarrow Y_\alpha$.
\end{Lemma}

\begin{proof}
Take $(V_0,\dots,V_n) \in X^\lambda_\mu$.
We need to show that $\langle f_1,\dots,f_h \rangle \subseteq V_m$.
We have that $(x^\lambda)^l (V_n) \subseteq V_m$
by the definition of $l$. 
The image of $(x^\lambda)^l$ is the span
of the basis vectors $f_i$
labelling boxes in rows $> l$ in the Young diagram of $\lambda$.
Because of our choice of labelling
(recall (\ref{piccy2})) this 
is exactly $\langle f_1,\dots,f_h\rangle$.
Hence $\langle f_1,\dots,f_h \rangle \subseteq V_m$.
\end{proof}

Now we can establish (\ref{s3}) for any $r > k-h$.
Clearly we can assume that $r \leq k < d$. Set $\delta := (1^r) \in \Par_{k,d-k}$
and
observe that 
$s_\delta(x_1,\dots,x_k)=e_r(x_1,\dots,x_k).$ 
So in view of Lemma~\ref{tiger} we just need to show
that $j^*(s^*(\sigma_\delta)) = 0$.
Letting $v:Y_\alpha \hookrightarrow \Gr_{k,d-k}$ be the inclusion, we
have from Lemma~\ref{sq2} that $s \circ j = v \circ \bar s$.
So $j^*(s^*(\sigma_\delta)) = \bar s^* (v^*(\sigma_\delta))$
and we are done if we can show that $v^*(\sigma_\delta) = 0$.
But this is clear as the classes $\{v_*([Y_\gamma])\:|\:\gamma \subseteq \alpha\}$
form a basis for $H_*(Y_\alpha\cfs)$ and $\delta\not\subseteq \alpha$.

Finally we can complete the proof of Theorem~\ref{mt1}.
We have shown
so far that there is a unique homomorphism
$\bar\psi:C^\lambda_\mu \rightarrow H^*(X^\lambda_\mu\cfs)$
making the diagram (\ref{diag2}) commute.
Moreover $\bar\psi$ is surjective since
$j^*$ is surjective by Corollary~\ref{cd}.
It remains to observe that $\dim C^\lambda_\mu \leq |\Col^\lambda_\mu|$
by the spanning part of Theorem~\ref{bt} established in $\S$\ref{sab},
while $\dim H^*(X^\lambda_\mu\cfs) = |\Col^\lambda_\mu|$
according to Corollary~\ref{cd}. Hence
$\bar\psi$ is actually an isomorphism, and Theorem~\ref{mt1}
is proved. 
At the same time, we have shown that
$\dim C^\lambda_\mu = |\Col^\lambda_\mu|$, so the vectors
$\{h(\mathtt{T})\:|\:\mathtt{T} \in \Col^\lambda_\mu\}$
which are already known to span $C^\lambda_\mu$ must actually be linearly independent. Thus we have also proved Theorem~\ref{bt}.

\begin{Remark}\rm
Theorem~\ref{mt1} also holds over the integers.
More precisely, 
if one replaces every occurence of $\C$ with $\Z$
in the definitions 
of the algebras $C^\lambda$ and $C^\lambda_\mu$ in the introduction, 
then the statement of
Theorem~\ref{mt1}
with $\C$ replaced by $\Z$ is still true.
So $H^*(X^\lambda_\mu,\Z)$
is isomorphic to the 
analogue of $C^\lambda_\mu$ 
over the integers.
Moreover, the latter algebra is a free $\Z$-module
with basis given by the elements listed in Theorem~\ref{bt}.
To see all this, we just note that 
the proof of Corollary~\ref{cd} works over $\Z$, 
giving in particular that 
$H^*(X^\lambda_\mu,\Z)$ is a free $\Z$-module.
Also the proof of the spanning part of Theorem~\ref{bt} given in $\S$\ref{sab}
is valid over $\Z$. Finally, the analogue of the 
key relation (\ref{s1}) over $\Z$, which is all 
that is needed to adapt the argument
in the previous paragraph, follows
from the relation established over $\C$
thanks to the freeness of $H^*(X^\lambda_\mu,\Z)$.
\end{Remark}

\end{document}